\documentclass[reqno,oneside,11pt]{amsart}

\usepackage{hyperref}
\usepackage{geometry}
\usepackage{longtable}
\usepackage[ansinew]{inputenc}
\usepackage{graphicx}
\usepackage{amsmath}
\usepackage{amsthm}
\usepackage{amssymb, color}%
\usepackage[numbers, square]{natbib}
\usepackage{mathrsfs}
\usepackage{bbm}
\usepackage{tikz}
\usepackage{mathptmx}

\usepackage{amsmath,amsthm,amsfonts,amssymb}
\usepackage{dsfont,extarrows,enumerate}
\usepackage{fullpage}
\usepackage{xcolor}

\renewcommand{\baselinestretch}{1.2} 
\numberwithin{equation}{section}

\makeatletter
\def\@tocline#1#2#3#4#5#6#7{\relax
  \ifnum #1>\c@tocdepth 
  \else
    \par \addpenalty\@secpenalty\addvspace{#2}%
    \begingroup \hyphenpenalty\@M
    \@ifempty{#4}{%
      \@tempdima\csname r@tocindent\number#1\endcsname\relax
    }{%
      \@tempdima#4\relax
    }%
    \parindent\z@ \leftskip#3\relax \advance\leftskip\@tempdima\relax
    \rightskip\@pnumwidth plus4em \parfillskip-\@pnumwidth
    #5\leavevmode\hskip-\@tempdima
      \ifcase #1
       \or\or \hskip 1em \or \hskip 2em \else \hskip 3em \fi%
      #6\nobreak\relax
    \dotfill\hbox to\@pnumwidth{\@tocpagenum{#7}}\par
    \nobreak
    \endgroup
  \fi}
\makeatother

\theoremstyle{plain} \newtheorem{theorem}{Theorem}[section]
\theoremstyle{plain} \newtheorem{proposition}[theorem]{Proposition}
\theoremstyle{plain} \newtheorem{lemma}[theorem]{Lemma}
\theoremstyle{plain} \newtheorem{corollary}[theorem]{Corollary}
\theoremstyle{definition} \newtheorem{definition}[theorem]{Definition}
\theoremstyle{definition} 
\theoremstyle{remark} \newtheorem{remark}[theorem]{Remark}
\theoremstyle{remark} \newtheorem{example}[theorem]{Example}

\renewcommand{\P}{\mathbf P}
\newcommand{\Prob}[1]{\P\left\{#1\right\}}
\newcommand{\E}{\mathbf E}

\newcommand{\R}{\mathbb R}

\newcommand{\NN}{\mathbb N}

\newcommand{\Sphere}{\mathbb{S}^{d-1}}

\newcommand{\sA}{\mathcal{A}}
\newcommand{\sB}{\mathcal{B}}
\newcommand{\sC}{\mathcal{C}}

\newcommand{\sM}{\mathcal{M}}

\newcommand{\sL}{\mathcal L}

\newcommand{\sF}{\mathcal F}
\newcommand{\sK}{\mathcal K}

\newcommand{\fv}{\mathfrak{f}}
\newcommand{\Fv}{\boldsymbol{\mathfrak{f}}}

\renewcommand{\subset}{\subseteq}

\newcommand{\one}{\mathbbm{1}}

\newcommand{\dodn}{\overset{{\rm d}}\longrightarrow}

\DeclareMathOperator{\card}{card}

\DeclareMathOperator{\conv}{\mathrm{conv}\,}
\DeclareMathOperator{\cl}{\mathrm{cl}\,}

\newcommand{\bhoper}[2]{{\rm conv}_{#1}(#2)}
\DeclareMathOperator{\cn}{cn}
\DeclareMathOperator{\Int}{Int}
\DeclareMathOperator{\piv}{\mathfrak{p}}

\DeclareMathOperator{\exo}{exo}

\renewcommand{\epsilon}{\varepsilon}
\newcommand{\eps}{\epsilon}
\renewcommand{\emptyset}{\varnothing}

\usepackage{fancybox}
\newlength{\querylen}
\begin{document}
\title[Facial structure of strongly convex sets]{Facial structure of strongly convex sets generated by random samples}

\author{Alexander Marynych}
\address{Faculty of Computer Science and Cybernetics, Taras Shevchenko National University of Kyiv, Kyiv, Ukraine}
\email{marynych@unicyb.kiev.ua}

\author{Ilya Molchanov}
\address{Institute of Mathematical Statistics and Actuarial Science,
  University of Bern, Alpeneggstr. 22, 3012 Bern,  Switzerland}
\email{ilya.molchanov@stat.unibe.ch}

\begin{abstract}
  The $K$-hull of a compact set $A\subset\R^d$, where $K\subset \R^d$
  is a fixed compact convex body, is the intersection of all
  translates of $K$ that contain $A$.  A set is called $K$-strongly
  convex if it coincides with its $K$-hull.  We propose a general
  approach to the analysis of facial structure of $K$-strongly convex
  sets, similar to the well developed theory for polytopes, by
  introducing the notion of $k$-dimensional faces, for all
  $k=0,\dots,d-1$.  We then apply our theory in the case when
  $A=\Xi_n$ is a sample of $n$ points picked uniformly at random from
  $K$.  We show that in this case the set of $x\in\R^d$ such that
  $x+K$ contains the sample $\Xi_n$, upon multiplying by $n$,
  converges in distribution to the zero cell of a certain Poisson
  hyperplane tessellation.  From this results we deduce convergence in
  distribution of the corresponding $f$-vector of the $K$-hull of
  $\Xi_n$ to a certain limiting random vector, without any
  normalisation, and also the convergence of all moments of the
  $f$-vector.
\end{abstract}

\keywords{Ball convexity, random convex bodies, $f$-vector, strongly
 convex set, facial structure, $K$-strong convexity, zero cell of a
 Poisson tessellation}

\subjclass[2010]{Primary: 60D05; secondary: 52A22,  52B05}

\maketitle

\vspace{-1.5cm}

\renewcommand{\baselinestretch}{1.1}\normalsize
\tableofcontents
\renewcommand{\baselinestretch}{1.2}\normalsize

\begin{center}{\sc List of notation}\end{center}
Throughout this paper we use the following notation and notions:

\noindent
\begin{tabular}{cl}
  \multicolumn{2}{l}{Subsets of $\R^d$:}\\
  $B_{r}(x)$ & ~--~ the closed ball of radius $r>0$ centred at $x\in\R^d$,\\
  $\Sphere$ & ~--~ the centred unit sphere in $\R^d$.\\
  \multicolumn{2}{l}{Families of closed subsets of $\R^d$:}\\
  $\mathcal{C}^d$ & ~--~all compact sets in $\R^d$,  p.~\pageref{eq:cd_def},\\
  $\sK^d$ & ~--~all compact convex sets in $\R^d$,  p.~\pageref{eq:kd_def},\\
  $\sK^d_{0}$ & ~--~all compact convex sets in $\R^d$ containing the
                origin,  p.~\pageref{eq:kd_def_with_0},\\ 
  $\sK^d_{(0)}$ & ~--~all compact convex sets in $\R^d$ containing the
                  origin in the interior,
                  p.~\pageref{eq:kd_def_with_int_0}.\\ 
  \multicolumn{2}{l}{Operations on subsets of $\R^d$:}\\
  $K+L$ & ~--~the Minkowski sum of sets $K,L\in\mathcal{C}^d$,\\
  $K\ominus L$ & ~--~ the Minkowski difference of sets
                 $K,L\in\mathcal{C}^d$,  p.~\pageref{eq:mink_diff_def},\\
  $\conv(A)$ & ~--~the convex hull of $A\in\mathcal{C}^d$,\\
  $\bhoper{K}{A}$ & ~--~the $K$-hull of a given set $A\in\mathcal{C}^d$,
               p.~\pageref{eq:bh_def},\\ 
  $K^{o}$ & ~--~ the polar body to a given set $K\in\sK^d$,
            p.~\pageref{eq:polar_body_def}.\\ 
  \multicolumn{2}{l}{Topological operations in $\R^d$:}\\
  ${\rm cl}(A)$ & ~--~ the closure of $A\subset\R^d$ in the standard
                  topology on $\R^d$,\\ 
  $\Int (A)$ & ~--~ the set of all interior points of $A\subset\R^d$
               with respect to the standard topology on $\R^d$,\\ 
  $\partial A$ & ~--~ the boundary of $A\subset\R^d$ with respect to the
               standard topology on $\R^d$.\\  
  \multicolumn{2}{l}{Convex geometry:}\\
  $h(K,u)$ &~--~ the support function of $K\in\sK^d$ in direction
             $u\in\R^d\setminus\{0\}$,
             p.~\pageref{eq:supp_func_def},\\ 
  $H(K,u)$ &~--~ the supporting hyperplane of $K\in\sK^d$ in direction
             $u\in\R^d\setminus\{0\}$,
             p.~\pageref{eq:support_hyperplane_def},\\ 
  $F(K,u)$ &~--~ the support set of $K\in\sK^d$ in direction
             $u\in\R^d\setminus\{0\}$,  that is,  $F(K,u)=K\cap
             H(K,u)$,\\ 
  $N(K,x)$&~--~ the normal cone to $K\in\sK^d$ at $x\in K$,
            p.~\pageref{eq:normal_cone_def},\\ 
  $\tau(K,R)$ &~--~ the reverse spherical image of $R\subset\Sphere$
                for $K\in\sK^d$,  Eq.~\eqref{eq:tau_def},\\ 
  $\mathcal{F}(K)$ &~--~ the family of all faces of $K\in\sK^d$,  see
                   p.~\pageref{eq:face_def} for the definition of a
                   face,\\
  $S_{d-1}(K,\cdot)$ &~--~ the surface area measure of a convex body
                       $K\in\sK^d_0$.\\
                       
 \multicolumn{2}{l}{Probability and measures:}\\
  $\P$, $\E$ &~--~probability and expectation corresponding to the
               choice of a probability space $(\Omega,\mathfrak{F},\P)$,\\
  $V_d$ &~--~the Lebesgue measure on $\R^d$,\\
  $\mathcal{H}_{d-1}$ &~--~ the $(d-1)$-dimensional Hausdorff measure on $\R^d$,\\
  $\Xi_n$ &~--~random set $\{\xi_1,\dots,\xi_n\}$ of $n$ independent points
            uniformly distributed in $K\in\mathcal{K}^d$,\\
  $\mathcal{P}_K$ &~--~ the Poisson process on
                    $(0,\infty)\times\Sphere$ with intensity measure,
                    being the product of\\ 
             &\phantom{~--~} the constant $V_d(K)^{-1}$,  the Lebesgue
               measure $V_1$ on $(0,\,\infty)$ and $S_{d-1}(K,\cdot)$,
               $K\in\sK^d_{(0)}$, p.~\pageref{eq:ppp_def},\\ 
  $\Pi_K$  &~--~ the Poisson process on $\R^d\setminus\{0\}$ obtained
             as the image of $\mathcal{P}_K$ under the mapping\\ 
             &\phantom{~--~}  $(0,\,\infty)\times \Sphere \ni(t,u)\mapsto
               t^{-1}u\in\R^d\setminus\{0\}$.\\ 
\end{tabular}

\medskip
\noindent
A set $K\in\sK^d$ is called a convex body if
$\Int K\neq \varnothing$.  A convex body $K\in\sK^d$ is called
regular (or smooth) if the normal cone $N(K,x)$ is one-dimensional for
all $x\in\partial K$. A convex body $K\in\sK^d$ is called
strictly convex if $\partial K$ does not contain any proper segment.

\section{Introduction}
\label{sec:introduction}

Let $K$ be a convex body in $\R^d$, which contains the origin in its
interior. Consider a set $\Xi_n:=\{\xi_1,\dots,\xi_n\}$ composed of
$n$ independent copies of a random vector $\xi$ uniformly distributed
in $K$.  There is a substantial literature concerning probabilistic
properties of random polytopes obtained as convex hulls of $\Xi_n$,
see \cite[Chapter~8]{sch:weil08},
\cite{hug13,ReitznerCombinatorialStructure,reit10} and references
therein. As $n$ grows, the number of vertices of the polytope obtained
as the convex hull of $\Xi_n$ grows to infinity and one has to
properly normalise it in order to come up with a nontrivial limit. The
rate of growth heavily depends on smoothness properties of $K$. For
example, if $K$ has a sufficiently smooth boundary the average number
of vertices is of polynomial order ${\rm const}\cdot n^{(d-1)/(d+1)}$,
while if $K$ is itself a polytope, this quantity grows
logarithmically, as ${\rm const}\cdot(\log n)^{d-1}$, see
\cite{ReitznerCombinatorialStructure}.

A completely different behaviour of uniform samples on a half-sphere
was discovered in \cite{barany2017random}. Namely, it was shown that
the average numbers of vertices and facets of the spherical polytope,
obtained as the spherical convex hull of a uniform sample on the
half-sphere, converge to finite positive constants. This phenomenon
was explained in \cite{kab:mar:tem:19} by passing to stereographic
projections and establishing the convergence in distribution of the
properly scaled projected sample (regarded as a binomial point process
in the usual Euclidean space) to a certain Poisson point process,
whose conventional convex hull turned out to be a polytope with
probability one. This approach has clarified the aforementioned
convergence of averages, provided the identification of the limiting
constants and, moreover, has led to the proof of convergence in
distribution of the entire $f$-vector together with all power moments.
Further models exhibiting a very similar behaviour have been
considered in recent works
\cite{Baci+Bonnet+Thaele:2021,Kabluchko+Temesvari+Thaele:2020}.

A very similar phenomenon has been observed for ball hulls of random
samples. Recall that closed convex sets can be obtained as
intersections of closed half-spaces containing them. Replacing the
family of half-spaces with all translations of a ball yields the
definition of the ball hull of a set $A$ in Euclidean space as the
intersection of all balls of a fixed radius which contain
$A$. Accordingly, a set is called (strongly) ball convex if it
coincides with its ball hull, see \cite{bezdek13,bez:lan:nas:07} and
references therein.  It has been proved in \cite{fod:kev:vig14} that
the mean number of vertices and edges of the (unit) ball hull of a
uniform sample of points from the unit disk in $\R^2$ converge to the
constant $\pi^2/2$ as the size of the sample tends to
infinity. Remarkably, the latter constant coincides with the limiting
constant of the number of facets of the spherical polytope mentioned
in the previous paragraph.  This line of research was later on
augmented in \cite{fod:vig18} by showing converges of variances, still
in dimension two, and, later on, in \cite{fod19} extended to the
convergence of the mean value of the number of (appropriately defined)
facets in any dimension.  

Generalising the notion of strong ball convexity, it is possible to
replace a Euclidean ball with an arbitrary convex body $K$ and define
the $K$-hull of $A$ as the intersection of all translates of $K$ that
contain $A$. This concept (called the $K$-strong convexity) has been
intensively studied in \cite{bal:pol00,pol96} and accompanying
works. If $K$ is origin symmetric, it can be considered as the unit
ball in a Minkowski space and the $K$-hull of $A$ becomes the ball
hull of $A$ in a Minkowski space, see \cite{jah:mar:ric17} and
\cite{lan:nas:tal13}, the latter also includes the case of a not
necessarily origin symmetric $K$.
  
In this paper we study the $K$-hull, denoted by $Q_n$, of a random
sample $\Xi_n$ of $n$ independent and uniformly distributed points in
a convex body $K$. Then
$$
Q_n=\bigcap_{x\in\R^d:\,\Xi_n\subseteq x+K}(x+K)
$$
is the intersection of all translates of $K$ which contain $\Xi_n$. 

In dimension two and assuming that $K$ is sufficiently smooth, it is
straightforward to describe the facial structure of $Q_n$ in terms of
vertices and edges. In this case, the authors of \cite{fod:pap:vig20} show that the
expected number of vertices (equivalently, the expected number of
edges) of $Q_n$ converges to a finite value, however, the formula for
the constant is not correct.  In higher dimensions such a simple
decomposition of the boundary in terms of vertices and edges is no longer
available. In order to identify the facial structure of $K$-hulls, in
particular, of $Q_n$, in arbitrary dimension we develop a new concept
of the $f$-vector for a family of convex bodies, which boils down to
the usual $f$-vector of a polytope if these bodies are singletons.

The basic result establishes convergence in distribution of the
normalised Minkowski difference $X_n$ between $K$ and $Q_n$ to the zero
cell $Z$ in a hyperplane tessellation of $\R^d$ whose
directional intensity is determined by the surface area measure of
$K$. As a consequence, we prove the convergence in distribution of a
properly normalised intrinsic volumes of $X_n$, as $n\to\infty$,
together with all power moments. Furthermore, we show the convergence
in distribution (again, together with all power moments) of the vector
determining the facial structure of $Q_n$ to a random vector,
describing the facial structure of the zero cell $Z$.

In particular, it is shown that, if $K$ is strictly convex, regular,
origin symmetric, and also is a generating set (meaning that all
intersections of its translates are summands of $K$), then the
expected number of $(d-1)$-dimensional $K$-facets of $Q_n$ converges,
as $n\to\infty$, to $2^{-d}d!V_d(L)V_d(L^o)$, where $L$ is the
projection body of $K$ and $L^o$ is the polar body to $L$. This
is shown to be the special case of an analogous (but more involved)
formula proved for not necessarily origin symmetric $K$.

The paper is organised as follows. In
Section~\ref{sec:ball-convexity-with} we recall main concepts of the
$K$-strong convexity, set the notation and recall basic properties of
the $K$-hull operation. Section~\ref{sec:boundary-structure-k} extends
the concept of an $f$-vector to families of convex bodies. In
particular, we identify a general position for such families, which
extends the conventional general position concept for families of
singletons.  This general concept of $f$-vectors for families of
convex bodies is applied to $K$-hulls in
Section~\ref{sec:f-vector-ball}. The key idea is to identify the polar
set to the Minkowski difference between $K$ and a strongly convex set
$Q$ as the convex hull of the union of polars to translated copies of
$K$. We find conditions for this family to be in general position,
enabling us to identify their $f$-vectors.

Section~\ref{sec:k-convex-sets} deals with the setting of random
samples. The key results of the section, summarised in Theorem
\ref{thr:chull}, is a pair of dual limit relations for random convex
bodies related to $Q_n$. One is the already mentioned convergence in
distribution of the normalised Minkowski difference between $K$ and
$Q_n$, the $K$-hull of $\Xi_n$, to the zero cell in a certain
hyperplane tessellation of $\R^d$. The dual result provides
convergence in distribution of the corresponding polar bodies,
allowing us to deduce convergence in distribution of the $f$-vectors
in the subsequent section. Furthermore, we also obtain the convergence
in distribution of the intrinsic volumes and all their moments.

Finally, Section~\ref{sec:conv-texorpdfstr-vec} establishes the
convergence in distribution of the relevant $f$-vectors and also
convergence of all their power moments. The limit for the expected
number of facets has been explicitly calculated. If $K$ is origin
symmetric, this limit has a simple expression in trems of the volumes
of the projection body of $K$ and the polar projection body.

In the Appendix we prove three results that may be interesting for
their own sake. First, we show that a certain family of random convex
bodies pertained to the sample $\Xi_n$ is in general position with
probability one, akin to a similar (and easy) result for random
polytopes, saying that $d+1$ points sampled uniformly at random from
$K$ lie in a hyperplane with probability zero. Second, it is shown
that the convergence in distribution of convex hulls of unions of
binomial point processes on the family of convex bodies in $\R^d$
which contain the origin is equivalent to the convergence in
distribution of the whole processes in the vague topology.  Last but
not least we extend Schneider's result \cite{sch82}, concerning the
expected number of vertices of a zero cell $Z$, to not necessarily
even directional intensity measures.

\section{Ball convexity with respect to a convex body}
\label{sec:ball-convexity-with}

Denote by $\sC^d\label{eq:cd_def}$ the family of compact sets in
$\R^d$ equipped with the Hausdorff metric, and by
$\sK^d\label{eq:kd_def}$ the family of all compact convex sets in
$\R^d$.  Let $\sK_0^d\label{eq:kd_def_with_0}$ be the family of
compact convex sets which contain the origin and let
$\sK^d_{(0)}\label{eq:kd_def_with_int_0}$ be the family of convex
bodies $K$ (that is, compact convex sets with non-empty interior)
whose interior $\Int K$ contains the origin. Thus,
$$
\sC^d \supseteq \sK^d \supseteq \sK^d_0 \supseteq \sK^d_{(0)},
$$
$\sK^d$ is a closed subset of $\sC^d$, $\sK^d_0$ is a closed subset of
$\sK^d$, but $\sK^d_{(0)}$ is not closed in the Hausdorff metric.

For a set $L$ in $\R^d$ denote by $L+x$ its translation by $x\in\R^d$,
and by
\begin{displaymath}
  -L :=\{-x\in\R^d: x\in L\}
\end{displaymath}
its reflection with respect to the origin. Further, $\partial L$ is
the topological boundary of $L$.

For $K,L\in\sC^d$, their \emph{Minkowski sum} is 
\begin{displaymath}
  K+L:=\{x+y: x\in K,\;y\in L\},
\end{displaymath}
and the set
\begin{displaymath}
  \label{eq:mink_diff_def}
  K\ominus L:=\{x\in\R^d: L+x\subset K\}
\end{displaymath}
is called the \emph{Minkowski difference}, see, e.g.,
\cite[p.~146]{schn2}. The Minkowski difference is empty if $K$ does
not contain a translate of $L$. Note the following easy result.

\begin{lemma}
  \label{lemma:minus}
  For each $K,L\in\sC^d$,
  \begin{equation}
    \label{eq:1}
    K\ominus L=\{x: L\subset K-x\}=\bigcap_{y\in L}(K-y)=\bigcap_{y\in -L}(K+y). 
  \end{equation}
\end{lemma}

Fix a convex compact set $K\in\sK^d$.  For a compact set $A$ in $\R^d$,
define its \emph{$K$-hull} as
\begin{displaymath}
  \label{eq:bh_def}
  \bhoper{K}{A}:=\bigcap_{x\in\R^d:\,A\subset K+x} (K+x),
\end{displaymath}
so that $\bhoper{K}{A}$ is equal to the intersection of all translates of
$K$ which contain $A$.  If $A$ is not contained in any translate of
$K$, then its $K$-hull is set to be $\R^d$.  A set is said to be
\emph{$K$-strongly convex} if it coincides with its $K$-hull, see
\cite{bal:pol00}.  If $K$ is the Euclidean ball, then $\bhoper{K}{A}$ is
called the \emph{ball hull} of $A$ and a $K$-strongly convex set is
called \emph{ball convex}, see \cite{bezdek13,bez:lan:nas:07}.  If
$K\in\sK^d_{(0)}$ is (origin) symmetric, that is, $K=-K$, the $K$-hull
can be viewed as the ball hull in the Minkowski space with $K$ being
its unit ball, see \cite{jah:mar:ric17}.

Let $\cn_K(A)$ be the set of all $x$ such that $A\subset K+x$. By
Lemma~\ref{lemma:minus},
\begin{displaymath}
  \cn_K(A):=\{x\in\R^d: A\subset K+x\}=-(K\ominus A),
\end{displaymath}
and further
\begin{equation}
  \label{eq:7}
  \bhoper{K}{A}=\bigcap_{x\in\cn_K(A)} (K+x)=K\ominus
  (-\cn_K(A))=K\ominus(K\ominus A). 
\end{equation}
The following result shows that the mapping
$A\mapsto (K\ominus A)=-\cn_K(A)$ can be considered a dual to the
operation of taking $K$-hull of $A$. While the second statement is
known, see \cite[Lemma~3.1.10]{schn2}, we provide its short proof for
completeness.

\begin{proposition}
  \label{prop:minus}
  For all $A\in\sC^d$, we have 
  \begin{equation}
    \label{eq:2}
    K\ominus A=K\ominus \bhoper{K}{A},
  \end{equation}
  and $K\ominus A$ is $K$-strongly convex. Moreover, $Q\in\sK^d$ is
  $K$-strongly convex if and only if
  \begin{equation}
    \label{eq:8}
    Q=K\ominus(K\ominus Q). 
  \end{equation}
\end{proposition}
\begin{proof}
  Since $A\subset \bhoper{K}{A}$, we have 
  \begin{displaymath}
    K\ominus \bhoper{K}{A}\subset K\ominus A.
  \end{displaymath}
  Let $x\in K\ominus A$. Then $A\subset K-x$, so that
  $\bhoper{K}{A}\subset K-x$. Hence, $\bhoper{K}{A}+x\subset K$, meaning that
  $x\in K\ominus \bhoper{K}{A}$.
  By \eqref{eq:1}, the set $K\ominus A$ is $K$-strongly convex for all
  $A$.  The characterisation of $K$-strongly convex sets by
  \eqref{eq:8} follows from \eqref{eq:7}.
\end{proof}

A set $Q$ is called a \emph{summand} of $K$ if $K=Q+L$ for some set
$L$ in $\R^d$.  In this case, $K\ominus Q=L$ and $K\ominus L=Q$, hence
\eqref{eq:8} holds. Thus, each summand of $K$ is $K$-strongly
convex. The opposite implication holds for $K$, being a
generating set. Following \cite{bal:pol00} and \cite{pol96}, a
convex set $K\in\sK^d$ is called a \emph{generating set} if each
intersection of its translates is a summand of $K$. In this case, the
family of $K$-strongly convex sets coincides with the family of
summands of $K$. It is known that the Euclidean ball is a generating
set, and all convex bodies in dimension $d=2$ are generating sets, see
\cite[Theorem~2]{pol96} and \cite[Section~3.2]{schn2}.

A set $A$ is called $K$-spindle convex if $A$ contains
$\bhoper{K}{\{x,y\}}$ for all $x,y\in A$. In general, the $K$-strong
convexity implies the $K$-spindle convexity, and the inverse
implication holds if $K$ is a generating set, see \cite[Theorem~2]{kar04}.

Recall that $K\in\sK^d$ is called \emph{strictly convex}, if the
boundary $\partial K$ of  $K$ does not contain any proper segment.

\begin{lemma}
  \label{lemma:str-convex}
  If $K\in\sK^d$ is strictly convex, then all $K$-strongly convex sets
  are also strictly convex. In particular, for all $A\in\sC^d$, the set
  $K\ominus A$ is strictly convex or empty.
\end{lemma}
\begin{proof}
  Let $Q$ be $K$-strongly convex. The proof is particularly simple if
  $K$ is a generating set, so that $Q$ is a summand of $K$. Hence, if
  $Q$ has a proper segment on its boundary, then $K$ is no longer
  strictly convex, which is a contradiction.

  If $K$ is not necessarily a generating set, the proof follows the
  scheme of the proof of this fact for origin symmetric $K$ in
  \cite{jah:mar:ric17}. Assume that the segment $\conv\{x_1,x_2\}$ is
  a subset of $\partial Q$ for $x_1\neq x_2$. Then
  \begin{displaymath}
    x:=(x_1+x_2)/2\in\partial Q
    =\partial \bigg(\bigcap_{y\in\R^d:Q\subset K+y} (K+y)\bigg). 
  \end{displaymath}
  Therefore, there exists a sequence $(y_i)_{i\in\NN}$ such that
  $Q\subset K+y_i$ for all $i$ and the distance from $x$ to
  $\R^d\setminus (K+y_i)$ converges to zero as $i\to\infty$. Since the
  sequence $(y_i)_{i\in\NN}$ is necessarily bounded, assume without
  loss of generality that $y_i\to y_0$ as $i\to\infty$. Then
  $x\in(\partial K+y_0)$, because $x\in K+y_0$ and the distance from
  $x$ to $\R^d\setminus (K+y_0)$ is equal to zero.  Since
  $x_1,x_2\in K+y_0$, we necessarily have $x_1,x_2\in\partial K+y_0$,
  so that $\partial K$ contains a nontrivial segment, which is a
  contradiction.
\end{proof}

For a set $L$ in $\R^d$, its \emph{polar set} is defined by
\begin{displaymath}
  \label{eq:polar_body_def}
  L^o:=\{u\in\R^d: h(L,u)\leq 1\},
\end{displaymath}
where 
\begin{displaymath}
  \label{eq:supp_func_def}
  h(L,u):=\sup\{\langle u,x\rangle: x\in L\}
\end{displaymath}
is the \emph{support function} of $L$ and $\langle \cdot,\cdot\rangle$
is the inner product in $\R^d$.  If $L$ is convex, closed and contains the
origin in its interior, then $L^o$ is a convex body, and $(L^o)^o=L$,
see \cite[Theorem~1.6.1]{schn2}.

It is well known that the polar set to the intersection of convex
compact sets containing the origin is equal to the closed convex hull of
the union of their polar sets, see \cite[Theorem~1.6.3]{schn2}.  Thus,
\eqref{eq:1} yields
\begin{equation}
  \label{eq:K-A}
  (K\ominus A)^{o}=\cl \conv \left(\bigcup_{y\in A}(K-y)^{o}\right),
\end{equation}
see \cite[Theorem~1.6.3]{schn2} for finite $A$ with the general case
derived by similar arguments.  This representation will be of major
importance for us, since it leads to a description of the facial
structure of $K\ominus A=K\ominus \bhoper{K}{A}$ and, {\it mutatis
  mutandis}, of $\bhoper{K}{A}$ in Sections~\ref{sec:boundary-structure-k}
and \ref{sec:f-vector-ball} below.

\section{Facial structure of convex hulls of  collections of convex sets}
\label{sec:boundary-structure-k}

\subsection{General position concept}
\label{sec:gener-posit-conc}

A \emph{face} of a convex compact set $L\in\sK^d\label{eq:face_def}$
is a convex subset $F$ of $L$ such that $x,y\in L$ and $(x+y)/2\in F$
imply that $x,y\in F$. The family of all faces of $L$ is denoted by
$\sF(L)$.  Note that $L$ and $\varnothing$ are also faces.  All other
faces are called \emph{proper}, and the family of proper faces is
denoted by
\begin{displaymath}
  \sF'(L):=\sF(L)\setminus\{L,\varnothing\}.
\end{displaymath}
A dimension of a face $F\in\sF(L)\setminus\{\varnothing\}$ is the
dimension of the smallest affine subspace containing $F$.  Denote by
$\sF_k(L)$ the family of $k$-dimensional faces of $L$.  The relative
interiors of $F\in\sF(L)$ provide a disjoint decomposition of $L$, see
\cite[Theorem~2.1.2]{schn2}.  The topological boundary $\partial L$ is
the disjoint union of relative interiors of proper faces.

A $(d-1)$-dimensional affine subspace $H$ is said to be a
\emph{supporting hyperplane} of nonempty $L\in\sK^d$ if $H$ intersects
$L$ and $L$ is a subset of one of the two half-spaces bounded by $H$.
A set $E\subset L$ is called an \emph{exposed face} if there exists a
supporting hyperplane $H$ of $L$ such that $E=L\cap H$. Each exposed
face of $L$ is a face of $L$, and each proper face of $L$ is contained
in an exposed face of $L$, see \cite[p.~75]{schn2}.

Let $\sL:=\{L_i, i\in I\}\subset\sK^d$ be a collection of convex
compact sets, such that their convex hull
\begin{displaymath}
  \conv(\sL):=\conv \Big(\bigcup_{i\in I}\;L_i\Big)
\end{displaymath}
is a compact set. Recall that $\conv(\sL)$ is the set of all (finite)
\emph{convex combinations} $\sum_{j=1}^m \lambda_j x_j$ for $m\in\NN$,
$\lambda_1,\dots,\lambda_m\geq 0$, $\lambda_1+\cdots+\lambda_m=1$ and
$x_j\in L_{i_j}$, $i_j\in I$, $j=1,\dots,m$. A convex combination is
said to be \emph{positive} if all coefficients $\lambda_j$ are
strictly positive. By Carath\'{e}odory's theorem, see
\cite[Theorem~1.1.4]{schn2}, it suffices to let $m\leq d+1$.

Let $A$ be an arbitrary closed convex subset of some exposed face $F$
of $\conv(\sL)$.  Put
\begin{displaymath}
  \sM(\sL,A):=\{L\in\sL: L\cap A\neq\varnothing\}.
\end{displaymath}
Recalling that each proper face is a closed convex subset of some
exposed face, we see that $\sM(\sL,F)$ is well defined.
Furthermore, in this case we have
\begin{equation}
  \label{eq:F_rep}	
  F=\conv \bigg(\bigcup_{L\in \sM(\sL,F)}(F\cap L)\bigg),
  \quad F\in\sF\big(\conv(\sL)\big).
\end{equation}
Indeed, by Carath\'{e}odory's theorem
for every $x\in F$, there exist $m\leq d+1$ and
$\{L_1,\dots,L_m\}\subset\sL$, such that
$x$ is a positive convex combination of $x_i\in L_i$, $i=1,\dots,m$.
By definition of a face, this implies $x_i\in F$, and therefore
$x_i\in L_i\cap F$, for all $i=1,\dots,m$. Thus, every $x\in F$ can be
written as a convex combination of points from $F\cap L$ for
$L\in \sM(\sL,F)$, yielding
\begin{displaymath}
  F\subset \conv \bigg(\bigcup_{L\in \sM(\sL,F)}(F\cap L)\bigg).
\end{displaymath}
The converse inclusion is obvious, hence, \eqref{eq:F_rep} holds.

\begin{definition}
  \label{def:gp}
  The sets from $\sL$ are said to be in \emph{general position} if, for
  each closed convex subset $A$ of each exposed face of $\conv(\sL)$,
  the family $\sM(\sL,A)$ is finite, and
  \begin{equation}
    \label{eq:def_gp_formula}
    \sum_{L\in \sM(\sL,A)}\big(1+\dim(A\cap L)\big)\leq \dim(A)+1,
  \end{equation}  
  where $\dim$ denotes the affine dimension.
\end{definition}

Some examples of families $\sL$, which are in general position, and
which are not, are given on Figure~\ref{figure_gen_pos}.

\begin{figure}
\includegraphics[scale=0.35]{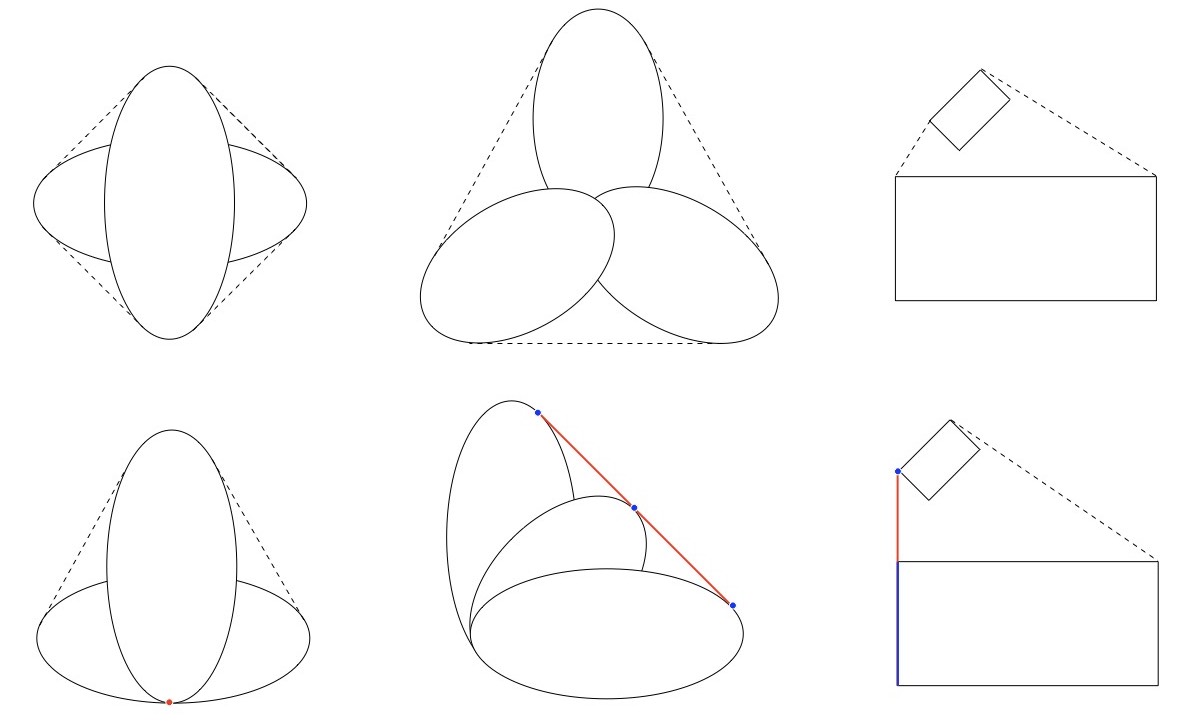}
\caption{First row: Three families of convex sets (two ellipses, three
  ellipses and two rectangles) in $\R^2$ in general position.  Second
  row: Three families of convex sets (two ellipses, three ellipses and
  two rectangles) which are not in general position.  In red: the sets
  $A$ for which \eqref{eq:def_gp_formula} is violated.  In blue: the
  sets $L\cap A$. The $\Fv$-vectors of the families depicted in the
  first row are $\Fv(\sL)=(2,1)$,
  $\Fv(\sL)=(3,3)$ and $\Fv(\sL)=(2,1)$, respectively.}
\label{figure_gen_pos}
\end{figure}

\begin{remark}
  \label{rem:equality}
  If $A=F$ is a face of $\conv(\sL)$, then the inequality in
  \eqref{eq:def_gp_formula} turns into the equality. This follows from
  \eqref{eq:F_rep} in view of the inequality
  \begin{displaymath}
    1+\dim(F)=1+\dim\bigg(\conv\Big(\bigcup_{L\in \sM(\sL,F)}(F\cap L)\Big)\bigg)
    \leq \sum_{L\in \sM(\sL,F)}\big(1+\dim(F\cap L)\big),
  \end{displaymath}
  where we have used the fact that the dimension of the convex hull of
  two sets of dimensions $m_1$ and $m_2$ is at most $m_1+m_2+1$.
\end{remark}

Assume that the sets in $\sL$ are in general position and
$F_m\in\sF_m(\conv(\sL))$ is an $m$-dimensional face of $\conv(\sL)$
for some $m=0,\dots,d-1$. By \eqref{eq:def_gp_formula} applied with
$A=F_m$, the cardinality of $\sM(\sL,F_m)$ is at most $m+1$, and
therefore every set in
\begin{displaymath}
  \sM(\sL):=\big\{\sM(\sL,F):\; F\in\sF'\big(\conv(\sL)\big)\big\}
\end{displaymath}
contains at most $d+1$ sets from $\sL$. For $k=0,\dots,d-1$, denote by
$M_k(\sL)$ the subset of $\sM(\sL)$ which consists of (unordered)
$(k+1)$-tuples.

\begin{definition}
  \label{def:facial_structure}
  Let $\sL$ be a family of convex bodies in general position. The
  elements of $M_k(\sL)$ are called \emph{$k$-dimensional faces} of the
  family $\sL$. The \emph{$\Fv$-vector} of the family $\sL$ is defined
  by the equality $\Fv(\sL):=\big(\fv_0(\sL),\dots,\fv_{d-1}(\sL)\big)$, where
  $\fv_k(\sL)$ is the cardinality of $M_k(\sL)$ counted without
  multiplicities, $k=0,\dots,d-1$.  As usual, vertices are
  $0$-dimensional faces of the family $\sL$, edges are $1$-dimensional
  faces of $\sL$, etc.
\end{definition}

It is important to stress that, in general, a face of the family $\sL$
is not a face of the convex compact set $\conv(\sL)$. For example, if
$\sL=\{L\}$ consists of a single convex body $L$, then $L$ is the
unique $0$-dimensional face of the family $\sL$.

The next lemma shows essentially that every face of a family $\sL$,
which is in general position, contains a vertex, that is, a
$0$-dimensional face of the family $\sL$.

\begin{lemma}
  \label{lemma:vertices}
  Let $\sL$ be a family of convex bodies in general position.  Then,
  for each $F\in\sF'(\conv(\sL))$ and $L\in \sM(\sL,F)$, there
  exists a $G\in\sF'(\conv(\sL))$ such that $L\cap G=G$.
\end{lemma}
\begin{proof}
  Assume that $L\cap F$ is a strict subset of $F$ since otherwise
  there is nothing to prove.  Let
  $\sM(\sL,F):=\{L_0,L_1,\dots,L_k\}$, where $L_0:=L$.  Denote
  by $F'$ the relative boundary of $F$. By \eqref{eq:F_rep}, we have
  \begin{equation}
    \label{eq:Fp}
    F=\conv\bigg(\,\bigcup_{i=0}^k (L_i\cap F')\bigg). 
  \end{equation}
  Indeed, by \eqref{eq:F_rep}, each point $x\in F'$ is the convex
  combination of points $y_i\in L_i\cap F$, $i=0,\dots,k$.  None of
  these points belongs to the relative interior of $F$. Thus, $F'$ is
  a subset of the right-hand side of \eqref{eq:Fp}.  Since $F$ is a
  subset of the convex hull of $F'$, we obtain \eqref{eq:Fp}.
  
  Denote $m_i:=\dim(L_i\cap F)$, $m'_i:=\dim(L_i\cap F')$, and
  $m:=\dim(F)$.  By Remark~\ref{rem:equality}, 
  \begin{displaymath}
    \sum_{i=0}^k (1+m_i)=m+1. 
  \end{displaymath}
  By \eqref{eq:Fp}, each point from $F$ is a convex combination of
  points from $L_i\cap F'$, $i=0,\dots,k$. Therefore,
  \begin{displaymath}
    \sum_{i=0}^k (1+m'_i)\geq m+1.
  \end{displaymath}
  Thus, $m_i=m'_i$ for all $i=0,\dots,k$, meaning that
  $L_i\cap F'\neq\varnothing$ and, in particular,
  $L\cap F'\neq\varnothing$. Hence, there exists
  \begin{displaymath}
    \tilde{F}\in \sF'(F)\subset\sF'(\conv(\sL)),
  \end{displaymath}
  such that $L\cap \tilde{F}\neq\varnothing$. Here, the second
  inclusion follows from \cite[Theorem~2.1.1]{schn2}.  The dimension of
  $\tilde{F}$ is at most $m-1$. By induction, reducing the dimension at
  each step, we find a proper face $G\in \sF'(\conv(\sL))$ such that
  $L\cap G=G$.
\end{proof}

\begin{corollary}
  \label{cor:combinatorial}
  Let $\sL$ be a family of convex bodies in general position.  Then
  \begin{displaymath}
    \fv_k(\sL)\leq \binom{\fv_0(\sL)}{k+1}, \quad k=0,\dots,d-1.
  \end{displaymath}
\end{corollary}
\begin{proof}
  Let $L\in \sM(\sL,F)$ for some $F\in\sF'(\conv(\sL))$. By
  Lemma~\ref{lemma:vertices}, there exists a $G\in\sF'(\conv(\sL))$ such
  that $L\cap G=G$. Then $\sM(\sL,G)$ contains $L$. Assume
  that $\sM(\sL,G)$ contains another set $L'\in\sL$, which is
  different from $L$. Then $L'\cap G\neq \varnothing$. By
  \eqref{eq:def_gp_formula},
  \begin{displaymath}
    \big(1+\dim(L\cap G)\big)+\big(1+\dim(L'\cap G)\big)\leq \dim(G)+1,
  \end{displaymath}
  which is a contradiction.  Thus, if a $(k+1)$-tuple of sets from
  $\mathcal{L}$ form a $k$-dimensional face of $\sL$, then each of
  these sets is a $0$-dimensional face of $\sL$, and the claim
  follows.
\end{proof}

\subsection{Families of strictly convex sets}
\label{sec:strictly-convex-sets}

For strictly convex sets, that is, for sets which do not contain any
proper segment on the boundary, the definition of general position can
be replaced by an equivalent one, which is much simpler.

\begin{proposition}
  \label{prop:def_gp_strictly_convex}
  If all sets in $\sL$ are strictly convex, Definition~\ref{def:gp} is
  equivalent to any of the following.
  \begin{itemize}
  \item[(i)] For all $m=0,\dots,d-1$, and each $m$-dimensional face
    $F_m$ of $\conv(\sL)$, the family $\sM(\sL,F_m)$ is finite and
    \begin{equation}
      \label{eq:def_gp_formula2}
      \sum_{L\in \sM(\sL,F_m)}\big(1+\dim(F_m\cap L)\big)= m+1.
    \end{equation}  
  \item[(ii)] For all $m=0,\dots,d-1$, and each $m$-dimensional face
    $F_m$ of $\conv(\sL)$, exactly $m+1$ sets among $\sL$ intersect
    $F_m$.
  \item[(iii)] For all $m=0,\dots,d-1$, and each $m$-dimensional exposed face
    $F_m$ of $\conv(\sL)$, exactly $m+1$ sets among $\sL$ intersect
    $F_m$.
  \end{itemize}
\end{proposition}
\begin{proof}
  Definition~\ref{def:gp} implies (i), since every face is a closed
  convex subset of some exposed face and in view of
  Remark~\ref{rem:equality}.

  Further, all non-empty sets of the form $F_m\cap L$ for
  $F_m\in\sF_m(\conv(\sL))$ and $m=0,\dots,d-1$ are singletons due to
  the imposed strict convexity. Hence $\dim(F_m\cap L)=0$ and
  \eqref{eq:def_gp_formula2} is equivalent to the fact that the
  cardinality of $\sM(\sL,F_m)$ is equal to $m+1$. This proves
  equivalence of (i) and (ii).  Since every exposed face is a face,
  (ii) implies (iii).

  It remains to prove that (iii) implies Definition~\ref{def:gp}. This
  is accomplished by contradiction. Assume that (iii) holds but there
  exists a closed convex subset $A$ of an exposed face $F$ such that
  $\dim(F)=m$, $\dim(A)=k$, $k=0,\dots,m$, $m=0,\dots, d-1$, and
  \begin{equation}
    \label{eq:proof_gp_contra}
    \card(\sM(\sL,A))=\sum_{L\in \sM(\sL,A)}\big(1+\dim(A\cap L)\big)\geq k+2.
  \end{equation}
  From condition (iii) it follows that \eqref{eq:def_gp_formula2}
  holds with $F_m=F$.  Therefore, $k<m$.  By the condition imposed in
  (iii) and strict convexity, $F$ intersects exactly $m+1$ sets in
  $\sL$ and these intersections are singletons, say
  $x_1,x_2,\dots,x_{m+1}$.  Moreover, at least $k+2$ among these
  singletons lie in $A\subset F$ by
  \eqref{eq:proof_gp_contra}. Without loss of generality, assume that
  $x_1,\dots,x_{k+2}\in A$. In view of \eqref{eq:F_rep},
  \begin{displaymath}
    F=\conv\{x_1,\dots,x_{m+1}\}\subset \conv\big(A\cup\{x_{k+3},\dots,x_{m+1}\}\big).
  \end{displaymath}
  Evaluating the dimension on both sides yields
  \begin{displaymath}
    1+m=1+\dim(F)\leq 1+\dim(A)+\sum_{j=k+3}^{m+1}\big(1+\dim(\{x_j\})\big)=1+k+(m-k-1)=m,
  \end{displaymath}
  which is a contradiction. The proof is complete.
\end{proof}

\begin{remark}
  Each point on the topological boundary of $\conv(\sL)$ is a positive
  linear combination of points from uniquely identified sets
  $L_{i_1},\dots,L_{i_{k+1}}\in\sL$. For strictly convex sets, the
  $k$-th component of the $\Fv$-vector can be equivalently defined as
  the number of different $(k+1)$-tuples in such collections for all
  points from the boundary of $\conv(\sL)$.
\end{remark}

\subsection{Examples}
\label{sec:examples}

\begin{example}[Polytopes]\label{example:poly}
  Let $\sL$ be a finite collection of singletons $L_i:=\{x_i\}$,
  $i=1,\dots,l$, so that its convex hull $\conv(\sL)$ is a
  polytope. Since singletons are strictly convex sets, the equivalent
  definition of the general position given by part (ii) of
  Proposition~\ref{prop:def_gp_strictly_convex} is applicable.  The
  general position for $\sL$ means that each $m$-dimensional face of
  this polytope contains exactly $m+1$ singletons for all
  $m=0,\dots,d-1$. Then $M_k(\sL)$ is the set of all $k$-dimensional
  faces of the polytope $\conv(\sL)$ (here we identify a face of a
  polytope with its extreme points), and $\Fv(\sL)$ is the usual
  $f$-vector\footnote{Throughout the paper we adopt the following
    convention: $\mathbf{f}$ is used to denote the usual $f$-vector of
    a polytope, while $\Fv$ is used to denote the $\Fv$-vector of a
    family of convex bodies in the sense of
    Definition~\ref{def:facial_structure}.} of the polytope
  $\conv(\sL)$. Thus, for singletons we have
  $\Fv\big(\{\{x_1\},\{x_2\},\dots,\{x_l\}\}\big)=\mathbf{f}\big(\conv\{x_1,x_2,\dots,x_l\}\big)$. Note
  that for sets of singletons the general position is usually defined
  by requiring that no $m+2$ points lie in an affine subspace of
  dimension $m$, see, for example,
  \cite[Lemma~4.1]{kab:mar:tem:19}. This complies with
  Definition~\ref{def:gp} imposed on $\sL$ and all its subfamilies.

  A similar situation occurs for a collection $\sL$ of sets
  $K_i:=[0,x_i]$, $i=1,\dots,l$, being segments with end-points at the
  origin and $x_i$. Then $\sL$ is in general position if and only if
  the points $\{0,x_1,\dots,x_l\}$ are in general position, and the
  $\Fv$-vector of $\sL$ is the $f$-vector of the polytope
  $\conv(\{x_1,x_2,\ldots,x_l\})$.
\end{example}

\begin{example}
  Let $\sL:=\{L\}$ consist of a single convex body $L$. Then
  $\Fv(\sL)=(1,0,\dots,0)$, no matter if $L$ is strictly convex or
  not.  This set $L$ is a vertex, no matter whether $L$ possesses
  higher-dimensional faces itself.
\end{example}

\begin{example}
  Let $\sL$ consist of all singletons from the boundary of the unit
  ball. Then $M_0(\sL)$ is an uncountably infinite collection of
  vertices of $\sL$,  that is,  faces of $\sL$ of dimension zero, and there
  are no faces of $\sL$ of dimension $1,2,\dots,d-1$. Thus,
  $\fv_0(\sL)=\infty$ and all other components of the $\Fv$-vector
  vanish. 
\end{example}

\begin{example}[Vertices]
  Assume that the sets in $\sL$ are in general position and let $A$ be
  an exposed point (that is, an exposed face of dimension zero) of $\conv(\sL)$.
  By \eqref{eq:def_gp_formula}, $\sM(\sL,A)$ consists of a single
  set, say, $L$. Thus, $L\in M_0(\sL)$ and, according to
  Definition~\ref{def:facial_structure}, $L$ is a vertex. Note
  that the facial structure of $L$ is not important at all here, in
  particular, $L$ may contain faces (in the usual sense) of arbitrary
  dimensions.
  For example, on the plane, let $L_1$ be the origin and $L_2$ be the
  segment $[(1,0),(0,1)]$, then $\fv_0(\sL)=2$ and both $L_1$ and
  $L_2$ are vertices. Furthermore, $\fv_1(\sL)=1$,
  since $M_1(\sL)=\{L_1,L_2\}$.  
\end{example}

\begin{example}
  Let $\sL$ be the collection of vertices of the square $[0,1]^2$ on
  the plane. Then $\Fv(\sL)=(4,4)$. Let $\sL'$ consist of
  the segment joining the origin and $(0,1)$ and the two other
  vertices of the unit square. Then $\Fv(\sL')=(3,3)$ is different
  from $\Fv(\sL)$, despite the two collections share the same convex
  hull. 
\end{example}

\begin{example}
  Assume that all sets from $\sL$ are strictly convex and the union of
  the sets from $\sL$ is a convex closed set. Then each face of
  $\conv(\sL)$ is also a face of at least one $L\in\sL$. If a point
  $x$ on the boundary of this union belongs to both $L_1$ and $L_2$
  from $\sL$, then the general position condition is violated. Hence,
  $\sL$ is in general position if and only if there exists a set
  $L_0\in\sL$ such that all members of $\sL\setminus\{L_0\}$ are
  subsets of the interior of $L_0$.
\end{example}

\section{\texorpdfstring{$\Fv$}{f}-vectors of
  \texorpdfstring{$K$}{K}-strongly convex  sets} 
\label{sec:f-vector-ball}

The aim of this section is to develop a notion of $k$-dimensional
faces of a $K$-strongly convex set in $\R^d$ of any dimension
$k\in\{0,1,\dots,d-1\}$. It has already been noted in the literature
that this task is quite nontrivial even in case of ball convex sets,
that is, when $K$ is a Euclidean ball, see a discussion at the
beginning of Section~6 in \cite{bez:lan:nas:07} and, in particular,
Example~6.1 therein.  We employ the approach developed in
Section~\ref{sec:boundary-structure-k}.

\subsection{Families of convex bodies associated with
  \texorpdfstring{$K$}{K}-hulls} 
\label{sec:famil-conv-bodi}

In the following, fix $K\in\sK^d_{(0)}$.  Let $A$ be a subset of the
interior of a translate of $K$. Introduce the family of sets
\begin{equation}
  \label{eq:9}
  \sL_A:=\big\{(K-y)^o: y\in A\big\}.
\end{equation}
By \eqref{eq:K-A}
\begin{equation}
  \label{eq:LA}
  (K\ominus A)^o =\cl \conv(\sL_A).
\end{equation}
If the sets in $\sL_A$ are in general position, it is possible to
define the $\Fv$-vector
$\Fv(\sL_A)=\big(\fv_0(\sL_A),\dots,\fv_{d-1}(\sL_A)\big)$ of the
family $\sL_A$.  If $A$ is a finite set, then the closure on the
right-hand side of \eqref{eq:LA} can be omitted, and $\Fv(\sL_A)$ has all
finite components.

\begin{example}
  Let $A:=\big\{(0,a),(0,-a)\big\}$ on $\R^2$, where $a\in(0,1)$, and
  let $K$ be the unit Euclidean disk. Then $\Fv(\sL_A)=(2,1)$.  If
  $A:=\big\{(0,0,a),(0,0,-a)\big\}$ in $\R^3$ and $K$ is the unit
  ball, then $\Fv(\sL_A)=(2,1,0)$.
\end{example}

In order to characterise the cases when $\sL_A$ is in general
position, we need several further concepts from convex geometry.
Recall that the \emph{normal cone} to a convex body $L\in\sK^d$ at
$x\in L$ is defined by
\begin{displaymath}
  \label{eq:normal_cone_def}
  N(L,x):=\big\{u\in\R^d\setminus\{0\}: x\in H(L,u)\big\}\cup\{0\},
\end{displaymath}
where 
\begin{displaymath}
  \label{eq:support_hyperplane_def}
  H(L,u):=\big\{x\in\R^d:\; \langle x,u\rangle = h(L,u)\big\}
\end{displaymath}
is the supporting hyperplane to $L$ with normal $u$.  The
normal cone is nontrivial only if $x\in\partial L$.  Denote by
\begin{displaymath}
  F(L,u):=L\cap H(L,u)
\end{displaymath}
the \emph{support set} of $L$ in direction $u$; this set is a
singleton $\{y\}$ if $L$ is strictly convex and we write $F(L,u)=y$ in
this case. The convex body $L$ is said to be \emph{regular} if
$N(L,x)$ is one-dimensional for all $x\in\partial L$, see
\cite[p.~83]{schn2}. For $L\in\sK^d_{(0)}$, this is equivalent to the
fact that the boundary of $L$ is $C^1$, see
\cite[Theorem~2.2.4]{schn2}.

If $F$ is a face of $L$, then $N(L,F):=N(L,x)$ for any $x$ in the
relative interior of $F$, see \cite[Section~2.2]{schn2}. Furthermore, the
\emph{conjugate face} to $F$ is defined as
\begin{displaymath}
  \widehat{F}:=\big\{x\in L^o: \langle x,y\rangle =1 \; \text{for all}\;
  y\in F\big\},
\end{displaymath}
see \cite[Section~2.1]{schn2}.  If $F$ is $(d-1)$-dimensional, then
$\widehat{F}$ arises as a solution of $d$ independent linear
equations and so is a singleton.  By \cite[Lemma~2.2.3]{schn2},
\begin{equation}
  \label{eq:6}
  N(L,F)=\mathrm{pos}\,\widehat{F} \cup\{0\},
\end{equation}
where $\mathrm{pos}\,\hat{F}$ is the positive hull of $\hat{F}$, that
is, the family of all linear combinations of points from $\hat{F}$
with nonnegative coefficients.

\begin{lemma}
  \label{lemma:gen-pos}
  Assume that $K\in\sK^d_{(0)}$ is a strictly convex regular convex body.
  If $A\subset \Int K$ is a finite set, then the family $\sL_A$ is in
  general position if and only if, for each $x\in\R^d$ such that
  $A\subseteq x+K$ and $A$ has a nonempty intersection with
  $\partial K+x$, this intersection is a finite set
  $\{y_1,\dots,y_k\}$ and $k$ is equal to the dimension of the convex
  hull of the union of the normal cones $N(K+x,y_i)$, $i=1,\dots,k$.
\end{lemma}
\begin{proof}
  For every $x\in\R^d$,  put
  \begin{displaymath}
    \mathcal{I}_A(x):=(\partial K+x)\cap A,
  \end{displaymath}
  and note that
  $-x\in \partial \big(\cap_{y\in\ A}(K-y)\big)=\partial (K\ominus A)$
  if and only if $A\subseteq x+K$ and
  $\mathcal{I}_A(x)\neq \varnothing$.  By \cite[Theorem~2.2.1]{schn2},
  for every $x\in\R^d$ such that $A\subseteq x+K$ and
  $\mathcal{I}_A(x)\neq \varnothing$, we have
  \begin{equation}
    \label{eq:NF}
    N(K\ominus A,-x)=N\bigg(\bigcap_{y\in A} (K-y),-x\bigg)
    =\conv \bigg(\bigcup_{y\in A} N(K-y,-x)\bigg)
    =\conv \bigg(\bigcup_{y\in \mathcal{I}_A(x)} N(K+x,y)\bigg).
  \end{equation}
  This relation will be of major importance in subsequent arguments.
 
  It follows from \eqref{eq:6} that regularity and strict convexity of
  $K$ imply regularity and strict convexity of the polar set $K^o$,
  see \cite[Remark~1.7.14]{schn2}.  Thus, for each $y\in \Int K$, the
  set $(K-y)^o$ is regular and strictly convex. By
  Proposition~\ref{prop:def_gp_strictly_convex}(iii), the family
  $\sL_A$ is in general position if and only if, for all
  $m=0,\dots,d-1$, each $m$-dimensional exposed face $F_m$ of
  $\conv(\sL_A)=(K\ominus A)^o$ intersects exactly $m+1$ sets from
  $\sL_A$.
  
  In view of \cite[Theorem~2.1.4]{schn2}, the second conjugate of each
  exposed face $F$ coincides with $F$.  Thus, by \eqref{eq:6}, the
  $m$-dimensional exposed faces of $(K\ominus A)^o$ are in one-to-one
  correspondence with their conjugate faces (exposed faces of
  $(K\ominus A)$) having $(m+1)$-dimensional normal cones.  By
  Lemma~\ref{lemma:str-convex}, $K\ominus A$ is strictly convex, so that
  these conjugate faces are singletons.  Thus, 
  the family $\sL_A$ is in general position if and only if, for every
  singleton $-x\in\partial(K\ominus A)$ such that
  $\dim N(K\ominus A,-x)=m+1$, $m=0,\dots,d-1$, exactly $m+1$ sets
  among $(\partial K-y)_{y\in A}$ contain $-x$.  Equivalently, the
  family $\sL_A$ is in general position if and only if for every
  $-x\in\partial(K\ominus A)$ such that $\dim N(K\ominus A,-x)=m+1$,
  $m=0,\dots,d-1$, we have $\card(\mathcal{I}_A(x))=m+1$. By
  \eqref{eq:NF}, the latter is equivalent to the
  following: for every $-x\in\partial(K\ominus A)$, that is, for every
  $x\in\R^d$ such that $A\subseteq x+K$ and
  $\mathcal{I}_A(x)\neq \varnothing$, we have
  \begin{displaymath}
    \dim \bigg(\conv \bigg(\bigcup_{y\in \mathcal{I}_A(x)} N(K+x,y)\bigg)\bigg)
    =\card(\mathcal{I}_A(x)).
  \end{displaymath}
  The proof is complete.
\end{proof}

\begin{remark}
  \label{rem:gen_pos_necessary}
  Letting $m=d-1$ in Lemma~\ref{lemma:gen-pos}, we obtain a necessary
  condition for the general position, saying that the cardinality of
  $A\cap (\partial K+x)$ is at most $d$ for all $x\in\R^d$.
\end{remark}

\begin{definition}
  \label{def:fvL}
  Let $Q:=\bhoper{K}{A}$, where $A$ is a subset of the interior of
  $K\in\sK^d_{(0)}$, such that the family $\sL_A$ defined at
  \eqref{eq:9} is in general position.  Then the \emph{$\Fv$-vector}
  $\Fv(Q)$ of $Q$ is defined as $\Fv(\sL_A)$.
\end{definition}

Note that in Definition~\ref{def:fvL} it is not feasible to work with
the family $\sL_Q$ defined by \eqref{eq:9} as the family of sets
$(K-y)^o$ for all $y\in Q$, because this family is not in general
position unless $Q$ is a singleton.  Indeed, otherwise, the boundary
of $Q$ contains a $(d-1)$-dimensional part of $\partial K-x$ for some
$x$ and so $Q$ intersects $\partial K-x$ at infinitely many points,
contrary to Lemma~\ref{lemma:gen-pos}.

\subsection{Facial structure of \texorpdfstring{$K$}{K}-hulls}
\label{sec:facial-structure-k}

A supporting $K$-sphere of $Q:=\bhoper{K}{A}$ is the set $x+\partial K$
such that $Q\subset x+K$ and $Q\cap (x+\partial
K)\neq\varnothing$. The set $Q\cap (x+\partial K)$ is said to be an
\emph{exposed $K$-face} of $Q$, see \cite{jah:mar:ric17}.
Note that these definitions in \cite{jah:mar:ric17} presume that $K$
is strictly convex and origin symmetric.

If $x+\partial K$ is an exposed $K$-face of $Q$ and
$Q\cap (x+\partial K)$ has a strictly positive $(d-1)$-dimensional
Hausdorff measure, then $Q\cap (x+\partial K)$ is called a
\emph{$K$-facet} of $Q$.  For $K$ being a Euclidean ball, this
definition was used in \cite{fod19} to describe the facial structure
of ball convex sets. Each $K$-facet contains at least $d$ affinely
independent points, but the inverse implication may fail, as the
following example shows. 

\begin{example}
  \label{ex:arc}
  Let $K$ be the unit Euclidean ball in $\R^3$, and let $y,z$ be two
  distinct points in the interior of $K$. The $K$-hull $Q$ of
  $A:=\{y,z\}$ is the intersection of all unit balls having $y$ and
  $z$ on the boundary. Such a ball $K+x$ intersects $Q$ along the arc
  of its great circle. While this arc contains 3 affinely independent
  points, its $2$-dimensional Hausdorff measure vanishes and so it
  is not a $K$-facet of $Q$, yet it is an exposed $K$-face of $Q$.
  In view of \eqref{eq:1}, $K\ominus Q=(K-y)\cap(K-z)$ and the polar
  to $K\ominus Q$ is the convex hull of the family
  $\sL_A:=\{(K-y)^o,(K-z)^o\}$. The boundary of $\conv(\sL_A)$ is
  composed of the parts of the boundaries of $(K-y)^o$ and $(K-z)^o$
  and an infinite number of one-dimensional faces, being segments
  joining points of $(K-y)^o$ and $(K-z)^o$. Then
  $\fv_{2}(Q)=\fv_{2}(\sL_A)=0$, which corresponds to the absence of
  $K$-facets in $Q$.
\end{example}

\begin{example}
  \label{ex:square}
  Let $K:=[-1,1]^2$ on $\R^2$, and let
  $A:=\{(a,0),(0,b),(-c,0),(-d,0)\}$ with $a,b,c,d\in(0,1)$.  Then,
  $Q=\bhoper{K}{A}=[-c,a]\times[-d,b]$, and 
  \begin{displaymath}
    K\ominus Q=[c-1,1-a]\times[d-1,1-b].
  \end{displaymath}
  The polar body to a polytope obtained as the intersection of the
  half-spaces $\{x:\langle x,u_i\rangle\leq t_i\}$, $i=1,\dots,m$,
  is equal to the convex hull of the points $t_i^{-1}u_i$, $i=1,\dots,m$.
  Hence,
  \begin{displaymath}
    (K\ominus Q)^o
    =\conv\big\{((1/(c-1),0),(0,1/(d-1)),(1/(1-a),0),(0,1/(1-b))\big\}. 
  \end{displaymath}
  Note that
  \begin{displaymath}
    (K+(a,0))^o=\conv\big\{(1/(1+a),0),(0,1),(1/(a-1),0),(0,-1)\big\}.
  \end{displaymath}
  Similar calculations for other points of $A$ show that $\sL_A$ is in
  general position and $\Fv(\sL_A)=(4,4)$. The $K$-hull $Q$ of $A$ is
  the intersection of $K+(a,b)$ and $K+(-c,-d)$, or other two
  translations of $K$.  In this case, $Q$ has two $K$-facets.
\end{example}
  
While Example~\ref{ex:square} shows that the number of $K$-facets of
$Q:=\bhoper{K}{A}$ (as defined following \cite{fod19} and
\cite{jah:mar:ric17}) may be different from the $(d-1)$-th component of
the $\Fv$-vector,  Lemmas~\ref{lemma:facets} and \ref{lemma:fd} below 
provide conditions ensuring that these quantities coincide.

For a convex body $K\in\sK^d$, denote by $S_{d-1}(K,\cdot)$ the
\emph{surface area measure} of $K$, see \cite[p.~214]{schn2}. The
surface area measure is a measure on the unit sphere $\Sphere$ in
$\R^d$ with the total mass being the surface area (that is, the
$(d-1)$-dimensional Hausdorff measure) of $K$.  If $K$ is regular,
then $S_{d-1}(K,A)$ is the surface area of the part of the boundary of
$K$ with unit normals belonging to the Borel set $A$ on the unit
sphere.

\begin{lemma}
  \label{lemma:facets}
  Assume that $K\in\sK^d_{(0)}$ is strictly convex and regular convex
  body, which is also a generating set.  Let $A\subset \Int K$ be a finite set such that $\sL_A$ given by \eqref{eq:9} is in general position.  If $Q$ is the $K$-hull of $A$,
  then the number of $K$-facets of $Q$ is equal to the number of
  $(d-1)$-dimensional faces of $\conv(\sL_A)$.
\end{lemma}
\begin{proof}
  Denote $R:=-(K\ominus A)=-(K\ominus Q)$, where the second equality
  follows from Proposition~\ref{prop:minus}. Since $A$ is a subset of
  the interior of a $K$, the set $R$ contains the origin in its
  interior and is also strictly convex by
  Lemma~\ref{lemma:str-convex}. In view of \eqref{eq:6}, each
  $(d-1)$-dimensional face $F$ of $-\conv(\sL_A)=R^o$ corresponds to a
  singleton $\{x\}$ (being the conjugate face to $F$) on the boundary
  of $R$ with the normal cone $N(R,x)$ having a non-empty
  interior. According to \cite[Theorem~2.1.4]{schn2}, the second conjugate
  face of $F$ is the smallest exposed face of $R$ containing $F$, and
  so the correspondence between $(d-1)$-dimensional faces of $R^o$ and
  points $x\in\partial R$ with normal cone $N(R,x)$ having a nonempty
  interior is one-to-one.  It remains to show that such points $x$ are
  in one-to-one correspondence with the $K$-facets of $Q$.  
  
  For $U\subset\Sphere$, the \emph{reverse spherical image} of $U$ is
  defined by
  \begin{equation}
    \label{eq:tau_def}
    \tau(K,U):=\bigcup_{v\in U} F(K,v),
  \end{equation}
  see \cite[p.~88]{schn2}.  Plugging here $U:=-(N(R,x)\cap \Sphere)$
  and calculating the $(d-1)$-dimensional Hausdorff measure, we obtain
  \begin{multline*}
    \mathcal{H}_{d-1}\Big(\big\{F(K+x,-u): u\in N(R,x) \big\}\Big)
    =\mathcal{H}_{d-1}\Big(\tau \big(K+x,-(N(R,x)\cap \Sphere) \big)\Big)\\
    =  S_{d-1}\Big(K+x,-\big (N(R,x)\cap \Sphere \big)\Big),
  \end{multline*}
  where the second equality follows from \cite[Eq.~(4.36)]{schn2}.
  Since $K+x$ is regular, the right-hand side is positive if and only
  if $N(R,x)\cap \Sphere$ has positive $(d-1)$-dimensional Hausdorff
  measure, which, in turn, is equivalent to the fact that $N(R,x)$ has 
  nonempty interior.  Thus, it remains to prove that
  \begin{displaymath}
    u\in N(R,x)\;\text{ if and only if }
    \;F(K+x,-u)\in\partial Q\;\text{ and }\;x\in R.
  \end{displaymath}
  Note that $u\in N(R,x)$ if and only if $x\in F(R,u)$, meaning that
  $x+v\notin R$ for all $v$ such that $\langle v,u\rangle
  >0$. Equivalently,
  \begin{equation}
    \label{eq:lemma:facets_proof2}
    u\in N(R,x)\;\text{ if and only if }\;K+x\supseteq A\text{ and }
    K+x+v\not\supseteq A\text{ for all }v\text{ such that }\langle v,u\rangle >0.
  \end{equation}

  We first show that $x\in F(R,u)$ implies $F(K+x,-u)\in \partial Q$.
  Since $F(K+x,-u)$ belongs to the boundary of $K+x$, this support
  point does not belong to the interior of $Q\subseteq K+x$.  Assume
  that $F(K+x,-u)\notin Q$ and so $F(K+x,-u)\notin K+y$ for some
  $y\in R$, $y\neq x$. Hence, 
  \begin{displaymath}
    A\subset (K+x)\cap (K+y):=L.
  \end{displaymath}
  From the strict convexity of $K$, we conclude $h(L,-u)<h(K+x,-u)$.
  Since $x\in H(R,u)$ and $y\in R$, we have that
  $\langle y-x,u\rangle \leq 0$.

  Since $K$ is a generating set, $K+x=L+W$ for a convex body
  $W$. Hence,
  \begin{displaymath}
    h(K+x,-u)=h(L,-u)+h(W,-u),
  \end{displaymath}
  and the support point $F(K+x,-u)$ is the sum of $F(L,-u)$ and
  $w:=F(W,-u)$. Thus,
  \begin{displaymath}
    h(K+x,-u)=h(L,-u)+\langle w,-u\rangle. 
  \end{displaymath}
  Since $h(L,-u)<h(K+x,-u)$, we must have $\langle
  w,-u\rangle>0$. Thus, $K+x\supset L+w$, so that $L\subset K+x+(-w)$
  with $\langle -w,u\rangle>0$. Therefore, $A\subset K+x-w$. This is a
  contradiction to \eqref{eq:lemma:facets_proof2}, since
  $\langle -w,u\rangle>0$. Thus, $F(K+x,-u)\in \partial Q$.

  In the other direction, assume that $y:=F(K+x,-u)\in\partial Q$ and
  $x\in R$. If $x+v\in R$ for some $v$ with $\langle v,u\rangle>0$,
  then $y\in Q\subseteq K+x+v$. Then
  \begin{displaymath}
    \langle y,-u\rangle\leq h(K+x+v,-u)<h(K+x,-u),
  \end{displaymath}
  contrary to the fact that $\langle y,-u\rangle=h(K+x,-u)$. Thus,
  $u\in N(R,x)$ in view of \eqref{eq:lemma:facets_proof2}.
\end{proof}

The next result establishes the equality between the number of
$(d-1)$-dimensional faces of $\conv(\sL_A)$ and
$\fv_{d-1}(\sL_A)$. Note that such a relationship is not feasible for
lower-dimensional faces. For example, $\sL_A$ might have a single
$0$-dimensional face generated by some $L\in\sL_A$, while
$\conv(\sL_A)$ has infinitely many 0-dimensional faces corresponding
to singletons on the boundary of $L$.

\begin{lemma}
  \label{lemma:fd}
  Assume that $K$ is strictly convex and regular, and let
  $A\subset \Int K$ be a finite set such that the family $\sL_A$ is in
  general position. Suppose further that for each set
  $\{x_1,\dots,x_d\}\subset A$ of cardinality $d$ which belongs to a
  $K$-facet of $Q:=\bhoper{K}{A}$, this set of points does not belong to
  any other $K$-facet of $Q$. Then $\fv_{d-1}(\sL_A)$ is equal to the
  number of $(d-1)$-dimensional faces of $\conv(\sL_A)$.
\end{lemma}
\begin{proof}
  We adapt the notation used in Lemma~\ref{lemma:facets}. By
  definition, each $(d-1)$-dimensional face of the family $\sL_A$, say, arising
  from the $d$-tuple $\big\{(K-x_1)^o,\dots,(K-x_d)^o \big\}$, corresponds to
  at least one face $F$ of $\conv(\sL_A)$. Note that
  $\{x_1,\dots,x_d\}$ is a subset of a $K$-facet of $Q$.

  Assume that there exists another face $F'$ of $\conv(\sL_A)$, which
  is also hit by each of $(K-x_1)^o,\dots,(K-x_d)^o$. The conjugate
  faces to $F$ and $F'$ are distinct singletons $\{x\}$ and $\{x'\}$
  such that $x,x'\in \partial(K\ominus A)$ and $\{x,x'\}\subset\partial K-x_i$,
  $i=1,\dots,d$. Thus, there are two translates $K-x$ and $K-x'$ which
  contain $\{x_1,\dots,x_d\}$ on the boundary, and so these points
  belong to different exposed $K$-facets of $Q$. This contradicts the
  assumption unless $F'=F$.
\end{proof}

\begin{example}
  \label{ex:two-points-2}
  Let $A:=\{x,y\}$ be a subset of the interior of the unit disk $K$ in
  $\R^2$. Then the $K$-hull $Q$ of $\{x,y\}$ is the intersection of
  two discs having $x$ and $y$ on the boundary. In this case
  $\Fv(\sL_A)=(2,1)$, while $Q$ has two $K$-facets. 
\end{example}

\section{\texorpdfstring{$K$}{K}-strongly convex sets generated by random samples}
\label{sec:k-convex-sets}

Fix a set $K\in\sK^d_{(0)}$.  Recall that
$\Xi_n:=\{\xi_1,\dots,\xi_n\}$ denotes a set of $n$ independent
points uniformly distributed in $K$. Motivated by the construction of
disk and ball polyhedra in \cite{fod19,fod:kev:vig14}, let
\begin{equation}
  \label{eq:3}
  Q_n:=\bigcap_{x\in\R^d:\,\Xi_n\subset K+x} (K+x)
\end{equation}
be the intersection of all translates of $K$ which contain
$\Xi_n$, that is, $Q_n=\bhoper{K}{\Xi_n}$. Note that it is possible to
replace $\Xi_n$ in \eqref{eq:3} by its (conventional) convex hull
$\conv(\Xi_n)$, so that $Q_n$ is the $K$-hull of $\Xi_n$ and also of
the polytope $P_n:=\conv(\Xi_n)$.

Further, let 
\begin{displaymath}
  X_n:=K\ominus\Xi_n=-\cn_K(\Xi_n). 
\end{displaymath}
By Proposition~\ref{prop:minus}, $X_n$ is $K$-strongly convex and by
formula \eqref{eq:1}
\begin{displaymath}
  X_n=K\ominus Q_n=\bigcap_{i=1}^{n}(K-\xi_i).
\end{displaymath}
Note that the interior of $X_n$ almost surely contains the origin, and
\eqref{eq:1} yields 
\begin{equation}
  \label{eq:5}
  X_n^o=\conv\left( \bigcup_{i=1}^n (K-\xi_i)^o\right). 
\end{equation}
This immediately implies that both $X_n$ and $X_n^o$ are random convex
sets, which are almost surely compact and have nonempty interiors, also
called \emph{random convex bodies}, see
\cite[Section~1.7]{mo1}. Furthermore, the interiors of both $X_n$ and
$X_n^o$ almost surely contain the origin, that is, $X_n$ and $X_n^o$
almost surely belong to $\sK^d_{(0)}$.
Simulations of $Q_n$ and $X_n$ for $d=2$, $K$ being a unit disk, and
$n=10,40,100$ are given on Figure~\ref{figure}.

\begin{figure}
\includegraphics[scale=0.44]{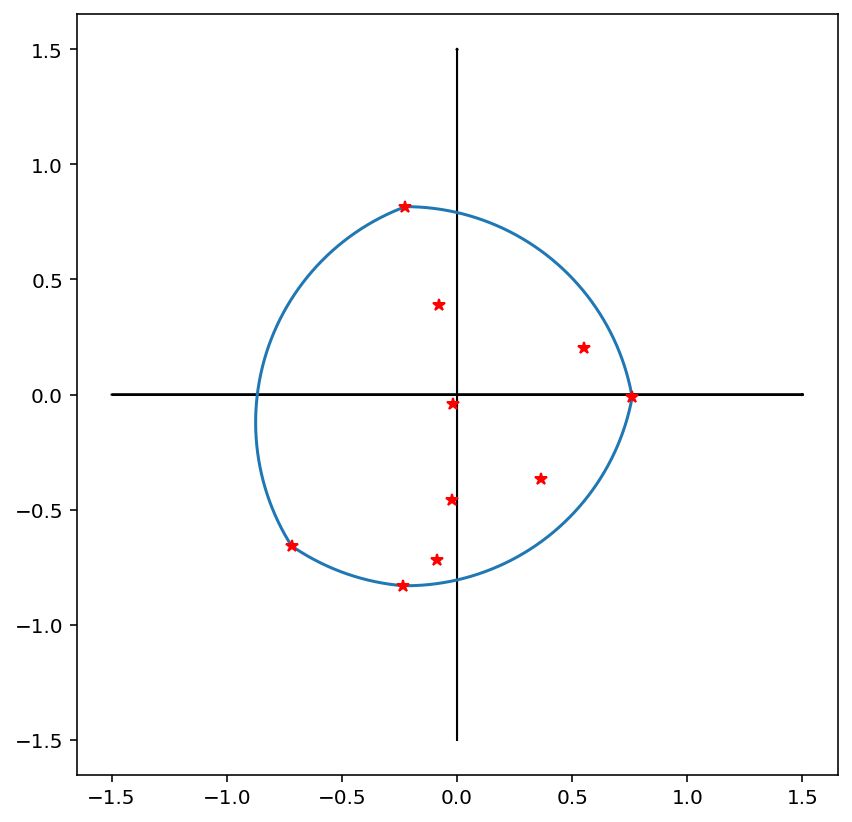}\quad\includegraphics[scale=0.44]{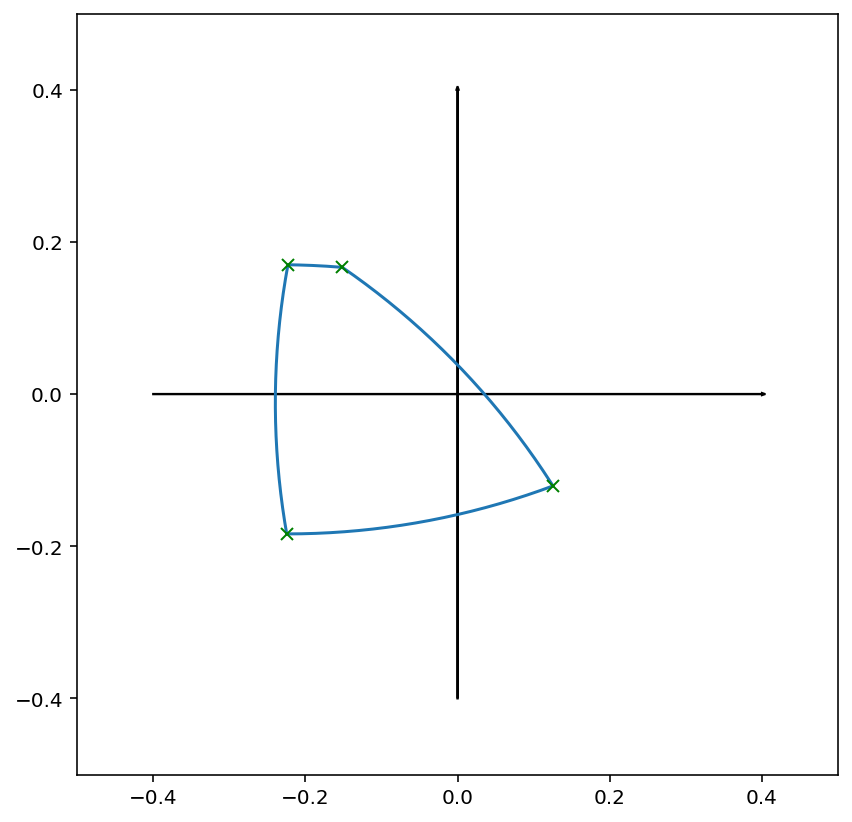}\\
\includegraphics[scale=0.44]{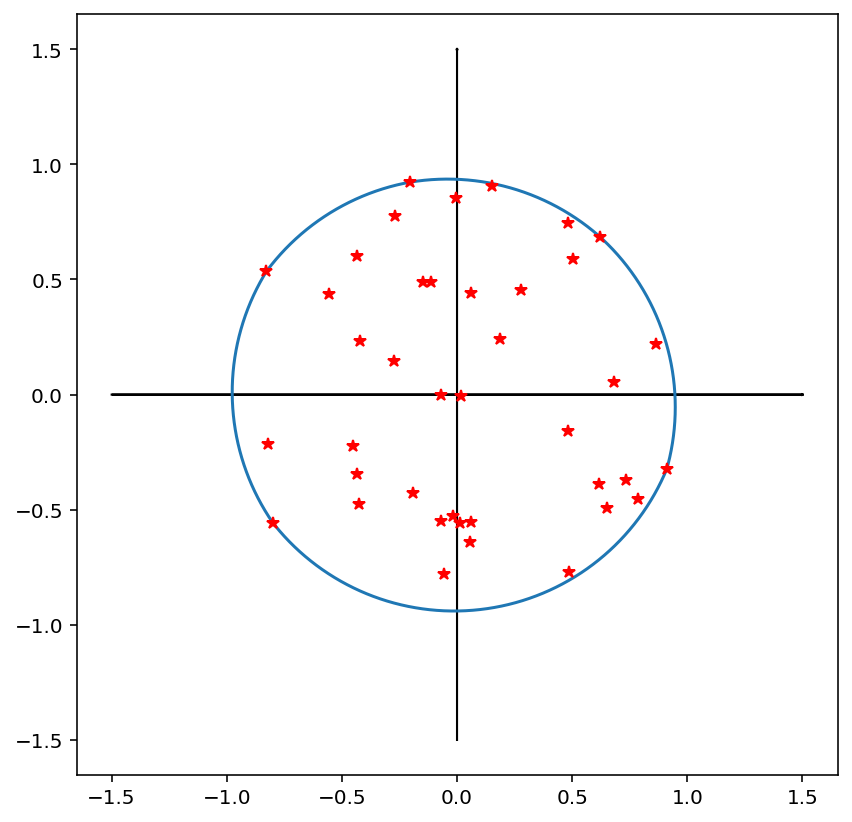}\quad\includegraphics[scale=0.44]{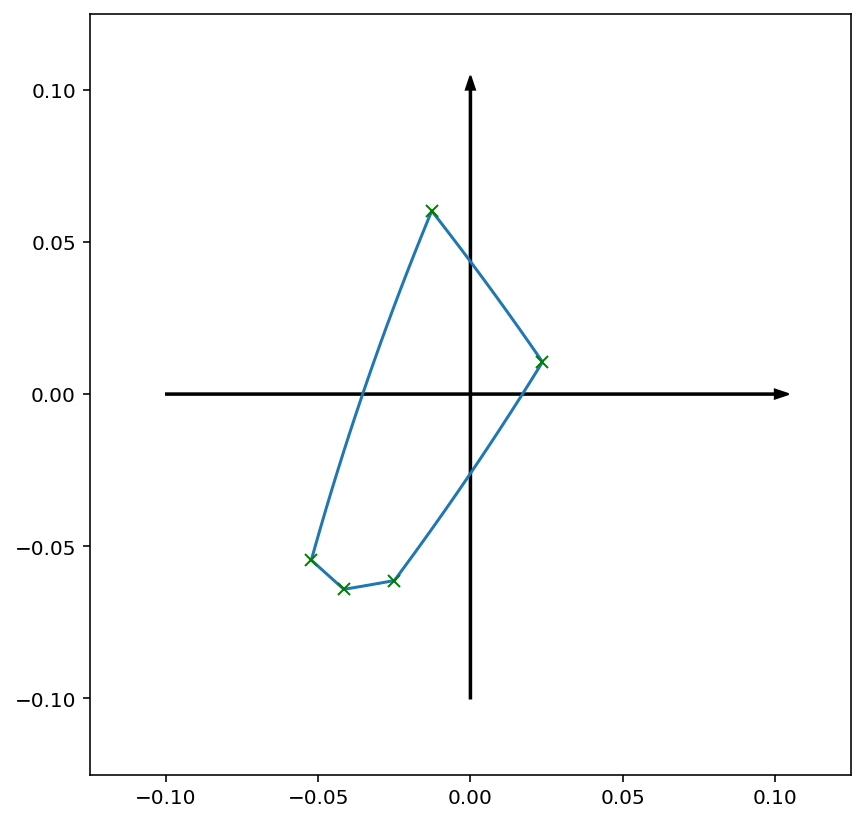}\\
\includegraphics[scale=0.44]{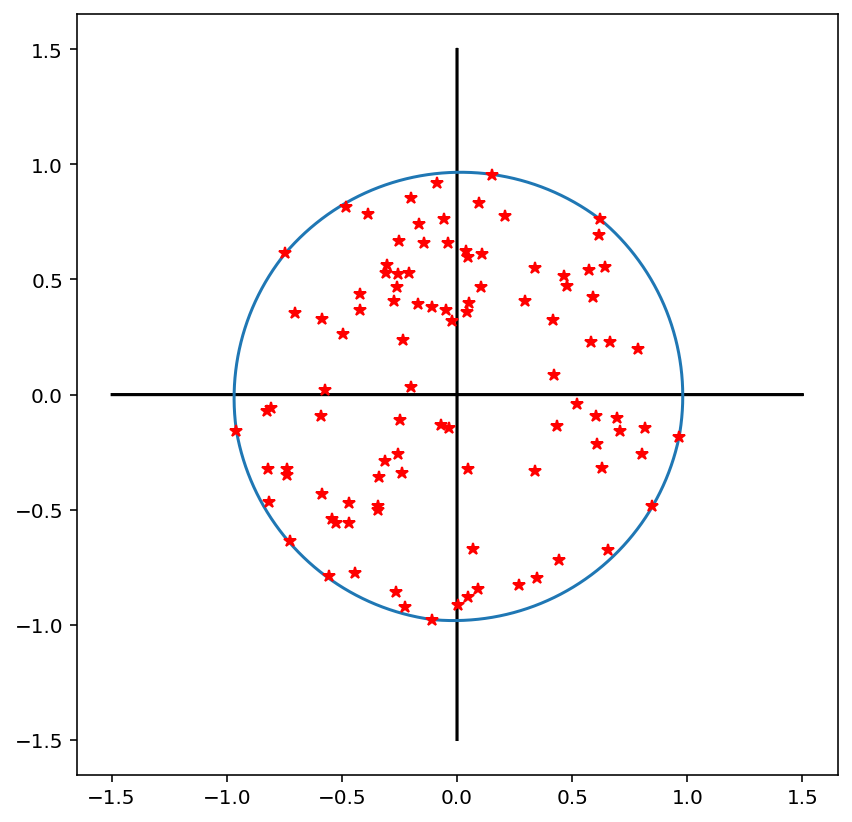}\quad\includegraphics[scale=0.44]{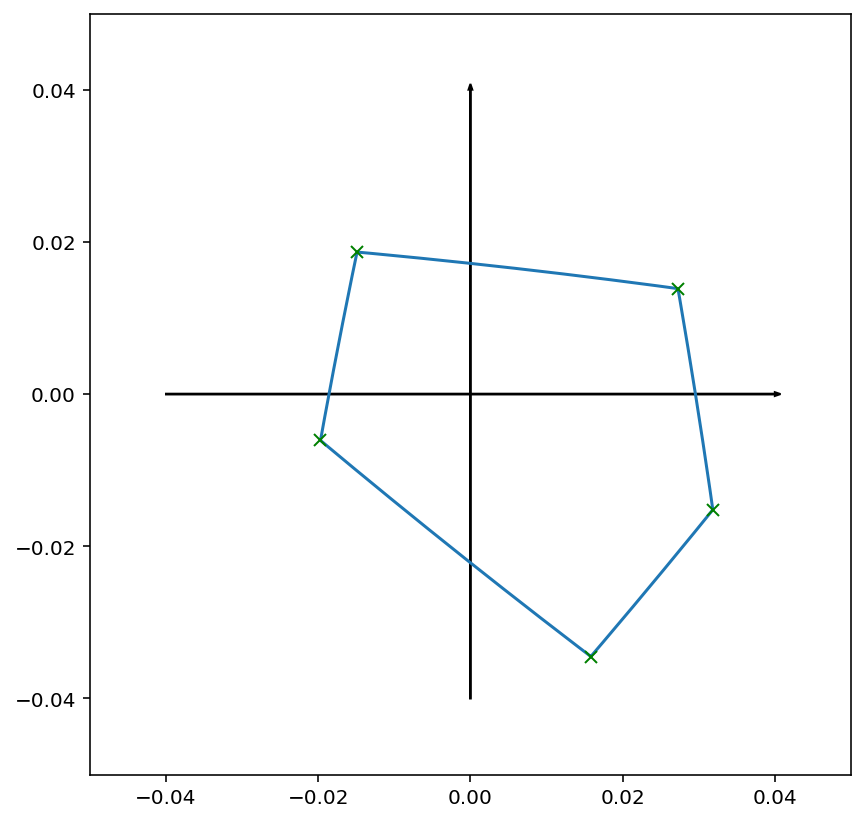}\\
\caption{$Q_{10}$ (top left) and $X_{10}$ (top right),  $Q_{10}$ is an intersection of four unit disks; $Q_{40}$ (middle left) and $X_{40}$ (middle right),  $Q_{40}$ is an intersection of five unit disks; $Q_{100}$ (bottom left) and $X_{100}$ (bottom right), $Q_{100}$ is an intersection of five unit disks.  Note a different scaling on the plots for $X_n$: in all three cases the lines connecting the 'vertices' are actually arcs of the unit circle,  that are rectified in the limit as suggested by Theorem~\ref{thr:chull}.}
\label{figure}
\end{figure}

\subsection{Convergence of the scaled Minkowski difference \texorpdfstring{$K\ominus \Xi_n$}{K-Xi}}\label{subsec:X_n_conv}

Theorem~\ref{thr:chull} below establishes the convergence in
distribution of $n^{-1}X_n^o$ and $nX_n$ as random convex sets, that
is, the weak convergence of the corresponding probability measures on
$\sK^d$ equipped with the Hausdorff metric, see
\cite[Section~1.8.2]{mo1}.  In order to formulate the result we need some
preparations.  Recall that $V_d$ stands for the Lebesgue measure on
$\R^d$, so that $V_d(K)$ is the volume of $K$.

Half-spaces in $\R^d$ which contain the origin in their interior are denoted
$H_{u}^-(t)$ and parametrised by $(t,u)\in (0,\infty)\times \Sphere$,
where $u$ is the unit outer normal vector and $t$ is the distance from
the origin to the boundary of the half-space.  Let
$\mathcal{P}_K:=\{(t_i,u_i),i\geq1\}\label{eq:ppp_def}$ be a Poisson
process on $(0,\infty)\times\Sphere$ with intensity measure $\mu$,
being the product of the Lebesgue measure on $(0,\infty)$ times the
constant $V_d(K)^{-1}$ and the measure $S_{d-1}(K,\cdot)$ on the unit
sphere (which is the surface area measure of $K$). Then
$\{H_{u_i}^-(t_i),i\geq1\}$ is a collection of half-spaces whose
boundaries are said to form a hyperplane process, inducing a
tessellation of $\R^d$, see \cite{sch:weil08}.  The measure
$S_{d-1}(K,\cdot)$ is called the directional component of the
tessellation. If $S_{d-1}(K,\cdot)$ is an even measure (which is the
case for an origin symetric $K$), then the tessellation is
stationary. Since all these half-spaces $H_{u_i}^-(t_i)$ contain the
origin in their interiors with probability one, their intersection is
not empty and is a random set denoted by $Z$.  Since the support of
the directional component $S_{d-1}(K,\cdot)$ of $\mu$ is not contained
in any closed hemisphere, $Z$ is almost surely bounded and, thus, is a
random compact convex set in $\R^d$. By the local finiteness of $\mu$,
$Z$ is almost surely a polytope called the \emph{zero cell} of the
Poisson hyperplane tessellation.

The polar set to $Z$ is the closed convex hull of the union of the
polar sets to $H_{u_i}^-(t_i)$, which are easily seen to be segments
$[0,t_i^{-1}u_i]$ with end-points at the origin and $t_i^{-1}u_i$. By
the mapping theorem for Poisson processes, the points
$\{t_i^{-1}u_i, i\geq1\}$ constitute a Poisson point process on
$\R^d\setminus\{0\}$ denoted by $\Pi_K\label{eq:Pi_def}$. The polar
set $Z^o$ is the convex hull of the points from $\Pi_K$, and
with probability one is a polytope which contain the origin in the
interior.

\begin{theorem}
  \label{thr:chull}
  The sequence of random convex bodies $(nX_n)_{n\in\NN}$ converges in
  distribution, as $n\to\infty$, to the zero cell $Z$ of the
  Poisson hyperplane tessellation introduced above.
  Furthermore, the random convex body $n^{-1}X^o_n$ converges in
  distribution to $Z^o$ as $n\to\infty$.
\end{theorem}
\begin{proof}
  We start with the second statement.  Since each $\xi_i$ almost
  surely belongs to the interior of $K$, $nX_n$ is a random compact
  convex set whose interior a.s.\ contains the origin. Hence,
  $n^{-1}X^o_n$ is indeed a random convex body.  By
  \cite[Theorem~1.8.14]{mo1}, it suffices to show that
  \begin{displaymath}
    \Prob{n^{-1}X^o_n\subseteq L}\to \Prob{Z^o\subseteq L}
    \quad \text{as }\; n\to\infty
  \end{displaymath}
  for all $L\in\sK^d$ and
  $\liminf_{n\to\infty} \Prob{n^{-1}X^o_n\subseteq L}\uparrow 1$ as
  $L\uparrow\R^d$. Since the interior of $X_n^o$ a.s.\ contains the
  origin, it suffices to assume that $L$ belongs to
  $\sK^d_{(0)}$. In view of \eqref{eq:5},
  \begin{align*}
    \Prob{n^{-1}X_n^o \subset L}
    &=\Prob{n^{-1}(K-\xi_i)^{o}\subset L\text{ for all }i=1,2,\dots,n}\\
    &=\Big(1-\Prob{n^{-1}(K-\xi_1)^{o}\not \subset L}\Big)^n\\
    &=\Big(1-\Prob{n^{-1} L^o \not \subset (K-\xi_1)}\Big)^n\\
    &=\Big(1-\Prob{\xi_1\notin K\ominus n^{-1} L^o}\Big)^n.
  \end{align*}
  By \cite[Theorem~1]{kid:rat06} applied with\footnote{Note that the
    authors of \cite{kid:rat06} use a slightly different definition of
    the Minkowski difference involving the set $L$ reflected
    with respect to the origin.}
  $C=K$, $A=K$, $P=B=W=\{0\}$, $Q=-(L^{o})$ and
  $\eps=n^{-1}$, we have
  \begin{align}
    \lim_{n\to \infty} n\Prob{\xi_1\notin K\ominus n^{-1}L^o}
    &=\lim_{n\to\infty} n \frac{V_d\big(K\setminus (K\ominus
      n^{-1}L^o)\big)}{V_d(K)}\notag \\
    &=\frac{1}{V_d(K)}\int_{\Sphere} h(L^o,u) S_{d-1}(K,{\rm d}u).
      \label{eq:kr_convergence}
  \end{align}
  Note that the set $K$ is gentle\footnote{The definition can be found
    on p.~107 in \cite{kid:rat06}.} by Proposition~1 in
  \cite{kid:rat06} because it is a convex body and is topologically
  regular, that is, coincides with the closure of its interior. Hence,
  \begin{displaymath}
    \Prob{n^{-1}X_n^o \subset L}
    \to \exp{\left(-\frac{1}{V_d(K)}\int_{\Sphere}
        h(L^{o} ,u) S_{d-1}(K,{\rm d}u)\right)}
    \quad \text{as }\; n\to\infty. 
  \end{displaymath}
  The right-hand side coincides with $\Prob{Z^{o}\subset L}$ because 
  \begin{align*}
    \Prob{Z^o\subset L}=\Prob{L^o\subset Z}
    &=\Prob{h(L^{o},u)\leq t\text{ for all }(u,t)\in\mathcal{P}_K}\\
    &=\exp\Big(-\mu\big(\{(u,t):h(L^{o},u)> t\}\big)\Big)\\
    &=\exp \bigg(-\frac{1}{V_d(K)}\int_{\Sphere} h(L^{o} ,u)
      S_{d-1}(K,{\rm d}u)\bigg).
  \end{align*}
  Since $Z^{o}$ is a.s.\ compact,
  $\liminf_{n\to\infty} \Prob{n^{-1}X^o_n\subseteq
    L}=\Prob{Z^{o}\subseteq L}\uparrow 1$ as $L\uparrow\R^d$.

  The convergence in distribution $nX_n\to Z$ follows from the
  continuous mapping theorem, since the transformation
  $A\mapsto A^o$ is Lipschitz on $\sK^d_{(0)}$, see
  \cite[Theorem~4.2]{mol15}.
\end{proof}

\begin{remark}
  In a recent preprint \cite{richey2020intersections}, a particular
  case of Theorem~\ref{thr:chull} was proved for $K$ being a unit ball
  and $\Xi_n$ replaced by a homogeneous Poisson process on $K$ with
  intensity $\lambda$ tending to infinity, see Theorem~1.2
  therein. Our Theorem~\ref{thr:chull} holds also in the Poisson setting.
\end{remark}

\subsection{Convergence of intrinsic volumes and their moments}

The intrinsic volumes $V_0,V_1,\ldots,V_{d}$ of a convex body $L$ are
defined by the Steiner formula
\begin{equation}
  \label{eq:steiner}
  V_{d}(L+rB_1(0))=\sum_{j=0}^{d}r^{d-j}\kappa_{n-j}V_j(L),\quad r\geq 0,
\end{equation}
where $\kappa_{j}=\pi^{j/2}/\Gamma(1+j/2)$ is the volume of the
$j$-dimensional unit ball, see Eq.~(4.8) in \cite{schn2}. It is well
known that all intrinsic volumes are continuous with respect to the
convergence in the Hausdorff metric. Thus, we immediately obtain from
Theorem~\ref{thr:chull} the following corollary.

\begin{corollary}
  Assume that $K\in\sK^{d}_{(0)}$. Then, for all $j=0,\ldots,d$,
  \begin{equation*}
    (V_{j}(nX_n))_{j=0,\ldots,d}=(n^jV_j(X_n))_{j=0,\ldots,d}
    \dodn (V_j(Z))_{j=0,\ldots,d} \quad\text{as }\; n\to\infty.
  \end{equation*}
\end{corollary}

With some additional efforts we can also deduce the convergence of all
power moments.

\begin{proposition}
  Assume that $K\in\sK^{d}_{(0)}$. Then, for all
  $j=0,\ldots,d$ and $m\in\NN$,
  \begin{displaymath}
    \lim_{n\to\infty}n^{mj}\E V_j(X_n)^m = \E V_j(Z)^m.
  \end{displaymath}
\end{proposition}
\begin{proof}
  We need to check that the sequence $(n^{mj}V^m_j(X_n))_{n\in\NN}$ is
  uniformly integrable, for all $m\in\NN$ and $j=0,\ldots,d$. It
  suffices to show that
  \begin{equation}\label{eq:vol_mom_proof1}
    \sup_{n\in\NN}\left(n^{mj}\E V_j(X_n)^m\right)<\infty,
  \end{equation}
  for all $m\in\NN$ and $j=0,\ldots,d$. By the Steiner formula
  \eqref{eq:steiner}, relation \eqref{eq:vol_mom_proof1} holds if
  \begin{displaymath}
    \sup_{n\in\NN} \E V_d\big(nX_n+B_1(0)\big)^m<\infty
  \end{displaymath}
  for all $m\in\NN$.  First, note that $B_1(y)$ intersects $nX_n$ if
  and only if there is a point $z\in B_1(y)$ such that
  $\Xi_n\subset K-n^{-1}z$. In this case $\Xi_n$ is also a subset of
  $K+n^{-1}B_1(y)$. Hence,
  \begin{align*}
    \E V_d\big(nX_n+B_1(0)\big)^m
    &=\E \int_{(\R^{d})^m} \one_{\{B_1(y_1)\cap nX_n\neq\emptyset\}}\cdots
      \one_{\{B_1(y_m)\cap nX_n\neq\emptyset\}}\, {\rm d}(y_1,\ldots, y_m)\\
    &\leq \E \int_{(\R^{d})^m} \one_{\{\Xi_n\subset K+n^{-1}B_1(y_1)\}}\cdots
      \one_{\{\Xi_n\subset K+n^{-1}B_1(y_m)\}}\, {\rm d}(y_1,\ldots, y_m)\\
    &=\int_{(\R^{d})^m} \Prob{\Xi_n\subset (K+n^{-1}B_1(y_1))\cap\cdots\cap
      (K+n^{-1}B_1(y_m))}\, {\rm d}(y_1,\ldots, y_m)\\
    &=\int_{(\R^{d})^m} \left(\frac{V_d\Big(K\cap
      \big(K+n^{-1}B_1(y_1)\big)\cap\cdots\cap
      \big(K+n^{-1}B_1(y_m)\big)\Big)}{V_d(K)}\right)^n\,
      {\rm d}(y_1,\ldots, y_m).
  \end{align*}
  Introducing the shorthand $K^r:=K+B_r(0)$ and making a change
  variables, we obtain
  \begin{displaymath}
    \E V_d\big(nX_n+B_1(0)\big)^m
    \leq \int_{(\R^{d})^m} n^{dm}\left(\frac{V_d\big(K\cap
        (K^{1/n}+y_1)\cap\cdots\cap
        (K^{1/n}+y_m)\big)}{V_d(K)}\right)^n\, {\rm d}(y_1,\dots,y_m).
  \end{displaymath}
  For each $c>0$,
  \begin{align*}
    &\hspace{-1cm}\int_{(\R^{d})^m} n^{dm}\one_{\{\|y_1\|< c/n\}}
      \left(\frac{V_d\big(K\cap (K^{1/n}+y_1)\cap\cdots\cap
      (K^{1/n}+y_m)\big)}{V_d(K)}\right)^n {\rm d}(y_1,\dots,y_m)\\
    &\leq \int_{(\R^{d})^m} n^{dm}\one_{\{\|y_1\|< c/n\}}
      \left(\frac{V_d\big(K\cap (K^{1/n}+y_2)\cap\cdots\cap
      (K^{1/n}+y_m)\big)}{V_d(K)}\right)^n {\rm d}(y_1,\dots,y_m)\\
    &=\int_{\R^d}\one_{\{\|y_1\|< c/n\}}{\rm d}y_1\int_{(\R^{d})^{m-1}} n^{dm}
      \left(\frac{V_d\big(K\cap (K^{1/n}+y_2)\cap\cdots\cap
      (K^{1/n}+y_m)\big)}{V_d(K)}\right)^n {\rm d}(y_2,\dots,y_m)\\
    &=\kappa_d c^d \int_{(\R^{d})^{m-1}} n^{d(m-1)}
      \left(\frac{V_d\big(K\cap (K^{1/n}+y_2)\cap\cdots\cap
      (K^{1/n}+y_m)\big)}{V_d(K)}\right)^n {\rm d}(y_2,\dots,y_m),
  \end{align*}
  where $\kappa_d$ is the volume of the unit ball. Iterating this
  bound yields the estimate
  \begin{displaymath}
    \E V_d\big(nX_n+B_1(0)\big)^m \leq c_0+\sum_{j=1}^m c_jI_n(j),
  \end{displaymath}
  where $c_0,c_1,\dots,c_m$ are nonnegative constants, which do not
  depend on $n$, and, for $j\in\NN$,
  \begin{equation}\label{eq:inm}
    I_n(j):=\\
    \int_{(\R^{d})^{j}}\one_{\{\|y_i\|\geq c/n, i=1,\dots,j\}} n^{dj}
    \left(\frac{V_d\big(K\cap (K^{1/n}+y_1)\cap\cdots\cap
        (K^{1/n}+y_j)\big)}{V_d(K)}\right)^n {\rm d}(y_1,\dots,y_j).
  \end{equation}
  Thus, it suffices to show that
  \begin{displaymath}
    \sup_{n\in\NN} I_n(j)<\infty
  \end{displaymath}
  for all $j\in\NN$. Note that the volume in the numerator vanishes
  whenever $\|y_i\|\geq M$ for some $i=1,\ldots,j$ and a suitable
  finite constant $M=M(K)$ because then $K\cap (K+y_i)=\varnothing$.

  Fix an arbitrary $C>0$. It is clear that there exists a
  $\delta_{K,C}\in(0,1)$, such that, for all $y\in\R^d$
  satisfying $\|y\|\geq C$ and all sufficiently large $n\in\NN$,
  \begin{displaymath}
    V_d(K\cap (K^{1/n}+y))\leq \delta_{K,C} V_d(K).
  \end{displaymath}
  This bound yields that, for sufficiently large $n\in\NN$, the part
  of the integral in \eqref{eq:inm} taken over the set
  $\{\|y_i\|\leq M,i=1,\ldots,j\}\setminus \{\|y_i\|\leq
  C,i=1,\ldots,j\}$ is bounded by
  ${\rm const}\cdot n^{dj}\delta_{K,C}^n$, which tends to zero as
  $n\to\infty$. Thus, it suffices to check that
  \begin{equation}
    \label{eq:inmt}
    \tilde{I}_n(m):=\int_{(\R^d)^m}\one_{\{\|y_i\|\in[c/n,C], i=1,\dots,m\}} n^{dm}
    \left(\frac{V_d\big(K\cap (K^{1/n}+y_1)\cap\cdots\cap
        (K^{1/n}+y_m)\big)}{V_d(K)}\right)^n {\rm d}(y_1,\dots,y_m),
  \end{equation}
  are uniformly bounded in $n\in\NN$, for each fixed $m\in\NN$ and a
  suitable choice of $c>0$ and $C>0$.

  It follows from \cite[Eq.~(10.1)]{schn2} that there exist $C>0$ and
  $a>0$ (possibly depending on $K$) such that
  \begin{displaymath}
    V_d\big(K\setminus (K+ru)\big)\geq  ar
  \end{displaymath}
  for all $u\in\Sphere$ and $r\in[0,C]$. Furthermore, from the Steiner
  formula \eqref{eq:steiner} it follows that there exists a $b>0$ such
  that, for all $n\in\NN$,
  \begin{displaymath}
    V_d(K^{1/n})-V_d(K)\leq b/n.
  \end{displaymath}
  Fix arbitrary $c>b/a$. Combining the above estimates yields that
  \begin{displaymath}
    V_d\big(K\setminus (K^{1/n}+ru)\big)
    \geq V_d\big(K\setminus (K+ru)\big)-\big(V_d(K^{1/n})-V_d(K)\big)
    \geq ar-b/n\geq a(r-c/n). 
  \end{displaymath}
  for all $u\in\Sphere$, $r\in [c/n,\,C]$ and sufficiently large
  $n\in\NN$.  Thus,
  \begin{displaymath}
    V_d\big(K\setminus (K^{1/n}+y_i)\big)\geq a(\|y_i\|-c/n),
  \end{displaymath}
  so that
  \begin{displaymath}
    \frac{V_d\big(K\cap (K^{1/n}+y_i)\big)}{V_d(K)}
    \leq 1-V_d(K)^{-1}a(\|y_i\|-c/n)=:1-a'(\|y_i\|-c/n),
  \end{displaymath}
  for all $y_i$ with $\|y_i\|\in[c/n,C]$. Note that the constants are
  adjusted in such a way that the right-hand side is always
  nonnegative on the domain $\{\|y_i\|\in[c/n,C],i=1,\dots,m\}$. By
  passing to polar coordinates on the right-hand side of
  \eqref{eq:inmt}, we see that
  \begin{align*}
    \tilde{I}_n(m)
    &\leq \int_{(\R^d)^m} \one_{\{\|y_i\|\in[c/n,C],i=1,\dots,m\}}
      n^{dm}\Big(1-a'\big(\max(\|y_1\|,\dots,\|y_m\|)-c/n\big)\Big)^n {\rm
      d}(y_1,\dots,y_m)\\
    &={\rm const}\cdot  \int_{[0,C-c/n]^m} n^{dm}
      \big(1-a'\max(r_1,\dots,r_m)\big)^n
      \big((r_1+c/n)\cdots (r_m+c/n)\big)^{d-1}\, {\rm d}(r_1,\dots,r_m)\\
    &=  {\rm const}\cdot \int_{[0,C-c/n]^m}\big(1-a' n^{-1}
      \max(s_1,\dots,s_m)\big)^n
      \big((s_1+c)\cdots (s_m+c)\big)^{d-1}\, {\rm d}(s_1,\dots,s_m)\\
    &\leq {\rm const}\cdot\int_{(0,\infty)^m}e^{-a'\max(s_1,\dots,s_m)}
      \big((s_1+c)\cdots (s_m+c)\big)^{d-1}\, {\rm d}(s_1,\dots,s_m)\\
    &\leq {\rm const}\cdot\int_{(0,\infty)^m}e^{-a' m^{-1} (s_1+\cdots+s_m)}
      \big((s_1+c)\cdots (s_m+c)\big)^{d-1}\, {\rm d}(s_1,\dots,s_m)\\
    &= {\rm const}\cdot\left(\int_0^\infty e^{-a' m^{-1}s}
      (s+c)^{d-1} {\rm d}s\right)^m
      <\infty 
  \end{align*}
  for all $n\in\NN$. The proof is complete.
\end{proof}

\begin{remark}
  From the H\"{o}lder inequality we immediately obtain the convergence
  of all mixed moments of the vector
  $(V_0(nX_n),V_1(nX_n),\ldots,V_d(nX_n))$, as $n\to\infty$, to the
  corresponding mixed moments of the vector of intrinsic volumes of
  the zero cell $Z$.
\end{remark}

\subsection{Convergence of point processes}

The sets involved in the closed convex hull operation on the
right-hand side of \eqref{eq:5} are independent copies of the random
compact set $(K-\xi)^o$, where $\xi$ is uniformly distributed in
$K$. These $n$ independent copies form a (binomial) point process
\begin{displaymath}
  \Psi_n:=\big\{(K-\xi_1)^o,\dots,(K-\xi_n)^o\big\}
\end{displaymath}
on the space $\sK^d_0$ of compact convex sets containing the
origin. By $n^{-1}\Psi_n$ we denote the scale transformation of
$\Psi_n$, which is the point process composed of the sets
$n^{-1}(K-\xi_i)^o$, $i=1,\dots,n$.

The Borel $\sigma$-algebra on $\sK^d_0$ is induced by the Hausdorff
metric. Let $\sB_0$ be the family of Borel
$\sA\subset\sK^d_0\setminus\{\{0\}\}$ such that the closure of $\sA$ in
$\sK^d_0$ does not contain $\{0\}.$\footnote{For typographical reasons we shall write in what follows $\sK^d_0\setminus\{0\}$ instead of $\sK^d_0\setminus\{\{0\}\}$.}  A Borel measure $\mu$ on
$\sK^d_{0}\setminus\{0\}$ is said to be locally finite if
$\mu(\sA)<\infty$ for all $\sA\in\sB_0$.  A sequence
$(\mu_n)_{n\in\NN}$ of locally finite measures on
$\sK^d_0\setminus\{0\}$ is said to converge \emph{vaguely} to a
locally finite measure $\mu$ if $\mu_n(\sA)\to\mu(\sA)$ as
$n\to\infty$ for all $\sA\in\sB_0$, which are continuity sets for the
limiting measure. Equivalently,
$\int f{\rm d}\mu_n\to \int f{\rm d}\mu$ for all bounded continuous
functions $f:\sK^d_0\setminus\{0\}\to\R$ which vanish in a
neighbourhood of $\{0\}$ (in the Hausdorff metric).

The convergence in distribution of random measures on $\sK^d$ is
understood with respect to the vague topology. This convergence
concept applies to random counting measures (or point
processes). Theorem~\ref{thr:pp} in the Appendix yields the following
result.

\begin{theorem}
  \label{th:convergence-pp}
  The sequence of point processes $(n^{-1}\Psi_n)_{n\in\NN}$ converges
  in distribution to the point process $\{[0,x]: x\in \Pi_K\}$ on
  $\sK^d_0$.
\end{theorem}

While the atoms of $n^{-1}\Psi_n$ belong to $\sK^d_{(0)}$, the atoms
of the limiting process do not, because they have empty
interiors. This phenomenon is due to the fact that the family
$\sK^d_{(0)}$ is not closed in $\sK^d$.

\section{Convergence of \texorpdfstring{$\Fv$}{f}-vectors of \texorpdfstring{$K$}{K}-hulls of random samples}
\label{sec:conv-texorpdfstr-vec}

Throughout this section we always assume that $K\in \sK^d_{(0)}$ is
{\it strictly convex and regular}.

\subsection{Limit theorems for the \texorpdfstring{$\Fv$}{f}-vector}\label{sec:main-results}

Lemma~\ref{lemma:gen-pos} yields that the finite family of convex bodies
$$
\sL_{\Xi_n}:=\{(K-\xi_k)^{o}:k=1,\dots,n\}
$$
is in general position with probability one.  Indeed, if we take
$x\in\R^d$ such that $\Xi_n\subseteq x+K$ and
$\partial K+x$ contains $\xi_{i_1},\ldots,\xi_{i_l}$, then, with
probability one, $l\leq d$ and the one-dimensional normal cones
$N(K+x,\xi_{i_1}),\ldots, N(K+x,\xi_{i_l})$ are linearly independent,
see Lemma~\ref{lem:gen_pos} in the Appendix.

Thus, with probability one the $\Fv$-vector of $Q_n=\bhoper{K}{\Xi_n}$ is
well defined and has all finite components.  In view of
Theorem~\ref{thr:chull}, it is natural to expect that
\begin{equation}
  \label{eq:f_vec_convergence}
  \Fv(Q_n)=\Fv(\sL_{\Xi_n})\dodn {\bf f}(Z^{o}) \quad\text{as }\; n\to\infty,
\end{equation}
and also the convergence of moments (possibly, under further
assumptions).  Recall that $Z^{o}=\conv(\Pi_K)$ is a.s.~a polytope and
${\bf f}(Z^{o})=(f_0(Z^{o}),f_1(Z^{o}),\ldots,f_{d-1}(Z^{o}))$ is the
$f$-vector of $Z^{o}$ in the usual sense.  We confirm this by proving
the following two results.

\begin{theorem}\label{thm:f_vector_convergence}
  Assume that $K\in \sK^d_{(0)}$ is {\it strictly convex and regular}.
  Then \eqref{eq:f_vec_convergence} holds.
\end{theorem}

\begin{theorem}\label{thm:f_vector_convergence_moments}
  Assume that $K\in \sK^d_{(0)}$ is {\it strictly convex and regular}.
  Then, for every $k=0,\ldots,d-1$ and every $m\in\NN$, we have
  \begin{equation}
    \label{eq:mom_convergence}
    \lim_{n\to\infty}\E \fv_k^m(Q_n)=\E f_k^m(Z^{o})<\infty.
  \end{equation}
\end{theorem}

\begin{remark}
  From the H\"{o}lder inequality we immediately obtain the convergence
  of all mixed moments of the vector $\Fv(Q_n)$, as $n\to\infty$, to
  the corresponding mixed moments of the $f$-vector of 
  $Z^o$.
\end{remark}

It is well known, see, for example, Corollary~2.13 and Theorem~2.14
from \cite{zieg95}, that the $f$-vector of $Z^o$ is the reversed
$f$-vector of $Z$, that is, $f_i(Z^o)=f_{d-i-1}(Z)$,
$i=0,\dots,d-1$. In particular, $\fv_{d-1}(Q_n)$ converges to the
number of vertices of the zero cell $Z$.

While it is genuinely difficult to calculate moments of the $f$-vector
for the zero cell $Z$ of an anisotropic tessellation (and even in the
isotopic case all first moments have been obtained only recently in
\cite{kab:20}), an explicit formula is exceptionally available for the
expectation of $f_0(Z)=f_{d-1}(Z^o)$. If the directional distribution
of the hyperplane tessellation is even, then
\begin{equation}
  \label{eq:f_0_Z_gen1}
  \E f_0(Z)=2^{-d}d! V_d(L)V_d(L^o),
\end{equation}
see \cite{sch82} and \cite[p.~376]{sch:weil08}. The convex body $L$ on
the right-hand side is determined by the directional distribution of
the hyperplane tessellation. In the special case, when this
directional distribution is $S_{d-1}(K,\cdot)$ and $K$ is origin
symmetric, the set $L$ is the \emph{projection body} of $K$, that is,
the support function of $L$ in direction $u\in\Sphere$ is equal to the
$(d-1)$-dimensional volume of the projection of $K$ onto the
hyperplane orthogonal to $u$. It is customary to denote the projection
body of $K$ by $\Pi K$, see \cite[Section~10.9]{schn2}, but we shall
use the notation $\piv K$ to avoid possible confusions with the
Poisson point process $\Pi_K$.

Note that the scaling parameter of
$S_{d-1}(K,\cdot)$ does not matter, since it cancels out in the
product of the volume of $L$ and its polar body. For not necessarily
origin symmetric $K$, the formula for $\E f_0(Z)$ seems to be
unavailable in the literature.  Using the same techniques as in
\cite{sch82}, we calculate this value in
Theorem~\ref{thm:zero_cell_vertexes} in the Appendix. The formula
reads as
\begin{equation}\label{eq:f_0_Z_gen2}
  \E f_0(Z)=\frac{1}{d} \int_{\Sphere}
  (h(\piv K,x))^{-d} J(x){\rm d}x,
\end{equation}
where $h(\piv  K,x)$ is the support function of the projection body $\piv K$ of $K$, see formula \eqref{app:eq:h_def} below,
\begin{equation}\label{eq:f_0_Z_gen3}
  J(x):=\int_{(\Sphere)^{d}}[v_1,v_2,\ldots,v_d]
  \one_{\{\langle v_1,x\rangle\geq 0,\ldots,\langle v_d,x\rangle\geq 0
    \}}
  S_{d-1}(K,{\rm d}v_1)\cdots S_{d-1}(K,{\rm d}v_d),\quad x\in\Sphere,
\end{equation}
and $[v_1,v_2,\ldots,v_d]$ is the volume of the parallelepiped spanned
by the vectors $v_1,\ldots,v_d$. For an origin symmetric $K$, the
quantity $J(x)$ does not depend on $x$ and is equal to the constant
\begin{displaymath}
  J:=2^{-d}\int_{(\Sphere)^{d}}[v_1,v_2,\ldots,v_d]
  S_{d-1}(K,{\rm d}v_1)\cdots S_{d-1}(K,{\rm d}v_d)
  =2^{-d}d!V_d(\piv K).
\end{displaymath}
In this case \eqref{eq:f_0_Z_gen2} reduces to \eqref{eq:f_0_Z_gen1} in
view of
\begin{displaymath}
  \frac{1}{d} \int_{\Sphere} (h(\piv K,x))^{-d}{\rm d}x=V_d((\piv K)^{o}).
\end{displaymath}


\subsection{Proof of Theorem~\ref{thm:f_vector_convergence}}
\label{sec:proof-theor-refthm:f-1}

Keeping in mind Theorem~\ref{th:convergence-pp}, we shall deduce
Theorem~\ref{thm:f_vector_convergence} from the following lemma which
establishes continuity of the $\Fv$-vector in the sense of
Definition~\ref{def:facial_structure} with respect to convergence of
families of convex bodies regarded as point processes on
$\sK^d_0\setminus\{0\}$.

\begin{lemma}
  \label{lemma:conv-tuples}
  Let $\sL^{(n)}:=\{L_i^{(n)}, i\geq1\}$, $n\in\NN_0$, be a sequence
  of locally finite point processes on $\sK^d_{0} \setminus\{0\}$.
  Suppose that all sets in $\sL^{(n)}$, $n\in\NN$, are strictly
  convex, and $\sL^{(n)}\to \sL^{(0)}$ in the vague topology on
  $\sK^d_0\setminus \{0\}$ as $n\to\infty$.  Further, assume the
  following:
  \begin{itemize}
  \item[(i)] The sets in $\sL^{(0)}$ are in general position.
  \item[(ii)] There exists a finite collection of singletons
    $\{x_1,\dots,x_M\}$ such that
    \begin{displaymath}
      \conv(\sL^{(0)})=\conv\{x_1,\dots,x_M\}\quad\text{and}
      \quad \{x_j\}=L^{(0)}_{r_j}\cap \partial\conv(\sL^{(0)}),\quad j=1,\dots,M,
    \end{displaymath}
    for a set of pairwise distinct indices $\{r_1,\dots,r_M\}$. 
  \item[(iii)] The convex hull $\conv(\sL^{(0)})$ contains the origin
    in the interior.
  \end{itemize}
  Then, for all sufficiently large $n\in\NN$,
  \begin{enumerate}
  \item[\textsc{(a)}] the sets in $\sL^{(n)}$ are in general position;
  \item[\textsc{(b)}] $\Fv(\sL^{(n)})=\Fv(\sL^{(0)})=\mathbf{f}\big(\conv(\sL^{(0)})\big)$;
  \item[\textsc{(c)}] $\fv_{d-1}(\sL^{(n)})$ coincides with the number of
    $(d-1)$-dimensional faces of $\conv(\sL^{(n)})$.
  \end{enumerate}
\end{lemma}
\begin{proof}
  We start by showing that the imposed assumptions imply that we can restrict our attention to finite subfamilies of $\{L_i^{(n)}, i\geq1\}$, $n\in\NN_0$.
  
  By (iii), there exits a ball $B_{2r}(0)$ such that
  \begin{equation}
    \label{eq:proof_continuity1}
    B_{2r}(0)\subset \Int\big(\conv(\sL^{(0)})\big).
  \end{equation}
  Since $\sL^{(0)}$ is locally finite on $\sK^d_0\setminus\{0\}$ and
  taking (ii) into account, the family $\sL^{(0)}$ contains only a
  finite number of sets which intersect $B^c_{r}(0)$, say,
  $L_1^{(0)},L_2^{(0)},\dots,L_l^{(0)}$, and
  $L_j^{(0)}\subset B_{r}(0)$ for all $j>l$.  By the imposed vague
  convergence, for all sufficiently large $n\in\NN$, the
  family $\sL^{(n)}$ contains exactly $l$ sets which intersect
  $B^c_{r}(0)$, say, $L_1^{(n)},L_2^{(n)},\dots,L_l^{(n)}$, and
  \begin{equation}
    \label{eq:proof_continuity11}
    L_j^{(n)}\to L_j^{(0)}\quad \text{as }\quad n\to\infty,\quad j=1,\dots,l,
  \end{equation}
  where the convergence is understood in the Hausdorff metric. 
  Furthermore, by the continuity of the convex hull operation, see
  \cite[Theorem~12.3.5]{sch:weil08},
  \begin{equation}
    \label{eq:proof_continuity2}
    \conv\bigg(\bigcup_{j=1}^{l}L_j^{(n)}\bigg)
    \to \conv\bigg(\bigcup_{i=j}^{l}L_j^{(0)}\bigg)=\conv(\sL^{(0)})
    \quad \text{as }\; n\to\infty.
  \end{equation}
  By \eqref{eq:proof_continuity1} and the convergence in
  \eqref{eq:proof_continuity2}, 
  $B_r(0)\subset \conv\left(\bigcup_{j=1}^{l}L_j^{(n)}\right)$ for
  all sufficiently large  $n$, and we conclude that
  \begin{equation}
    \label{eq:proof_continuity3a}
    \conv\big(\sL^{(n)}\big)=\conv\bigg(\bigcup_{i=1}^{l}L_i^{(n)}\bigg)
  \end{equation}
  and
  \begin{equation}
    \label{eq:proof_continuity3b}
    \conv\big(\sL^{(n)}\big)\to\conv\big(\sL^{(0)}\big)\quad
    \text{as }\quad n\to\infty. 
  \end{equation}
  In what follows we fix $l$ and assume that $n$ is picked so large
  that \eqref{eq:proof_continuity3a} holds.

  Recall that the upper limit of
  a sequence of sets is the set of limits for all convergent
  subsequences of points selected from these sets.  Furthermore,
  recall that the intersection operation is upper semicontinuous,
  meaning that the upper limit of intersections of two sets is a
  subset of the intersection of their upper limits, see
  \cite[Theorem~12.2.6]{sch:weil08}.

  \vspace{2mm}
	\noindent
	{\sc Proof of part \textsc{(a)}.} We argue by contradiction. If
  the sets in $\sL^{(n)}$ are not in general position for infinitely
  many $n\in\NN$, then, by the imposed strict convexity and
  Proposition~\ref{prop:def_gp_strictly_convex}(ii), for every such
  $n\in\NN$ there exist $m_n\in\{0,1,\dots,d-1\}$ and a face
  $F_{m_n}^{(n)}\in\sF_{m_n}\big(\conv(\sL^{(n)})\big)$ which is hit by
  $L^{(n)}_{i_{1,n}},\dots,L^{(n)}_{i_{k_n,n}}$, where $k_n\geq m_n+2$
  and $1\leq i_{1,n},\dots,i_{k_n,n}\leq l$. Since there are only
  finitely many possible values for $m_n$, $k_n$ and
  $i_{1,n},\dots,i_{k_n,n}$, we can pick $m\in\{0,1,\dots,d-1\}$,
  $k\geq m+2$, and $1\leq i_1,\dots,i_k\leq l$, such that
  $F_{m}^{(n)}\in\sF_{m}\big(\conv(\sL^{(n)})\big)$ is hit by
  $L^{(n)}_{i_1},\dots,L^{(n)}_{i_{k}}$ for infinitely many $n\in\NN$.

  As a face of $\conv(\sL^{(n)})$, the set $F_m^{(n)}$ is contained in
  an exposed face of $\conv(\sL^{(n)})$. By the definition of an
  exposed face, this means that
  \begin{equation}
    \label{eq:proof_cont_gen_pos1}
    F_m^{(n)}\subset \conv(\sL^{(n)})\cap H_n,
  \end{equation}
  for a supporting hyperplane $H_n$ of $\conv(\sL^{(n)})$. Let $u^{(n)}$
  be the unit normal vector to $H_n$, so that 
  \begin{displaymath}
    H_n=\Big\{x\in\R^d:\langle x,u^{(n)}\rangle
    =h\big(\conv(\sL^{(n)}),u^{(n)}\big)\Big\},
  \end{displaymath}
  where $h\big(\conv(\sL^{(n)}),u^{(n)}\big)$ is the support function of
  $\conv(\sL^{(n)})$ at $u^{(n)}$.
  By passing to a subsequence we can assume that $u^{(n)}\to u^{(0)}$
  as $n\to\infty$. This implies
  \begin{align}
    &\hspace{-0.5cm}\Big|h\big(\conv(\sL^{(n)}),u^{(n)}\big)
      -h\big(\conv(\sL^{(0)}),u^{(0)}\big)\Big|\notag \\
    &\leq \Big|h\big(\conv(\sL^{(n)}),u^{(n)}\big)
      -h\big(\conv(\sL^{(n)}),u^{(0)}\big)\Big|+\Big|h\big(\conv(\sL^{(n)}),u^{(0)}\big)
      -h\big(\conv(\sL^{(0)}),u^{(0)}\big)\Big|\notag\\
    &\leq \big\|\conv(\sL^{(n)})\big\|\cdot \|u^{(n)}-u^{(0)}\|
      +\Big|h\big(\conv(\sL^{(n)}),u^{(0)}\big)
      -h\big(\conv(\sL^{(0)}),u^{(0)}\big)\Big|\notag\\
    &\to 0\quad \text{as}\; n\to\infty,\label{eq:cont_lemma_proof1}
  \end{align}
  where $\|L\|:=\sup\big\{\|x\|:x\in L\big\}$, and we have used the Lipschitz
  property of support functions (see, e.g., \cite[Lemma~1.8.12]{schn2})
  and the fact that the convergence in \eqref{eq:proof_continuity3b}
  implies pointwise convergence of the corresponding support functions.
  The above argument also implies that 
  \begin{equation}
    \label{eq:proof_cont_gen_pos2}
    \limsup_{n\to\infty}\,H_n=\lim_{n\to\infty}H_n=H^{(0)}
    :=\Big\{x\in\R^d:\langle x,u^{(0)}\rangle
    =h\big(\conv(\sL^{(0)}),u^{(0)}\big)\Big\}
  \end{equation}
  is a supporting hyperplane of $\conv(\sL^{(0)})$.

  By the Blaschke selection theorem, the sequence of convex sets
  $(F_m^{(n)})_{n\in\NN}$ has a convergent subsequence. Passing to such
  a subsequence, assume that $F_m^{(n)} \to F$ as $n\to\infty$.  We
  claim that the limit $F$ is a subset of an exposed face of
  $\conv(\sL^{(0)})$. Indeed, letting $n\to\infty$ in
  \eqref{eq:proof_cont_gen_pos1} and using upper semicontinuity of the
  intersection yield that
  \begin{multline}\label{eq:proof_cont_gen_pos3}
    F=\limsup_{n\to\infty}F_m^{(n)}
    \subset \limsup_{n\to\infty}\;\big(\conv(\sL^{(n)})\cap H_n\big)\\
    \subset\limsup_{n\to\infty}\conv(\sL^{(n)})\,
    \cap\, \limsup_{n\to\infty}H_n=\conv(\sL^{(0)})\cap H^{(0)},
  \end{multline}
  where the last equality is a consequence of
  \eqref{eq:proof_continuity3a} and \eqref{eq:proof_cont_gen_pos2}.
  It remains to note that the dimension $m'$ of $F$ is at most $m$ and
  $F$ is hit by $L^{(0)}_{i_1},\dots,L^{(0)}_{i_{k}}$. The latter
  follows again from the upper semicontinuity of the intersection:
  \begin{displaymath}
    \varnothing\neq \limsup_{n\to\infty} \big(F_m^{(n)}\cap L_{i_j}^{(n)}\big)
    \subset \limsup_{n\to\infty} F_m^{(n)}\cap\,
    \limsup_{n\to\infty} L_{i_j}^{(n)}=F\cap L_{i_j},\quad j=1,\dots,k,
  \end{displaymath}
  in view of \eqref{eq:proof_continuity11}. This contradicts the
  assumption that the sets in $\sL^{(0)}$ are in general position
  because \eqref{eq:def_gp_formula} is violated for $A=F$. The proof
  of part~\textsc{(a)} is complete.

  \vspace{2mm}
  \noindent
  {\sc Proof of part \textsc{(b)}.} The second equality in
  part~\textsc{(b)} follows from the discussion in
  Example~\ref{example:poly} because assumptions (i) and (ii) imply
  that the $f$-vector of $\conv(\sL^{(0)})$ is completely determined
  by the set $\{L_{r_1},\dots,L_{r_M}\}$ which in turn can be replaced
  by the set of singletons $\{x_1,\dots,x_M\}$.  By the construction,
  the set of indices $\{r_1,\dots,r_M\}$ is a subset of
  $\{1,2,\dots,l\}$ in view of the equality in
  \eqref{eq:proof_continuity2}. Without loss of generality and in
  order to avoid towering indices let us assume that $l=M$ and $r_j=j$
  for $j=1,\dots,M$.
  
  Let us prove the first equality in (b). Fix $m\in\{0,1,\ldots,d-1\}$
  and put $s:=\limsup_{n\to\infty} \sum_{j=m}^{d-1}\fv_j(\sL^{(n)})$.
  By passing to a subsequence, assume that, for each $n\in\NN$, there
  exists a collection of $s$ different $(j+1)$-tuples,
  $j\in\{m,\ldots,d-1\}$ of sets from
  $\big\{L_1^{(n)},\dots,L_l^{(n)}\big\}$ such that for each of these 
  tuples there exists a $j$-dimensional face of $\conv(\sL^{(n)})$,
  intersecting each set from the corresponding tuple. Our first goal is
  to show that
  \begin{equation}
    \label{eq:cont_lemma_proof_limsup}
    \sum_{j=m}^{d-1}\fv_j(\sL^{(0)})\geq s
    =\limsup_{n\to\infty} \sum_{j=m}^{d-1}\fv_j(\sL^{(n)}),
  \end{equation}
  that is, for each aforementioned $(j+1)$-tuple there exists at least
  one face of $\conv(\sL^{(0)})$ of dimension at least $j$ and which
  is hit by the limiting sets of the chosen $(j+1)$-tuple. The proof
  of \eqref{eq:cont_lemma_proof_limsup} goes along similar lines as
  the proof of part~\textsc{(a)}. Namely, pick $j\in\{m,\ldots,d-1\}$,
  a $(j+1)$-tuple
  $\big\{L_{i_{1,n}}^{(n)},\dots,L_{i_{j+1,n}}^{(n)}\big\}$ and
  $F^{(n)}_{j}\in\mathcal{F}_j(\conv(\sL^{(n)}))$ such that
  $L_{i_{k,n}}^{(n)}\cap F^{(n)}_{j}\neq \varnothing$ for all
  $k=1,\ldots,j+1$. By passing to subsequences, it is possible to
  assume that $i_{k,n}=i_k$ for all $k=1,\ldots,j+1$. Thus, from now
  on assume that for all $n\in\NN$, each
  $L_{i_1}^{(n)},\dots,L_{i_{j+1}}^{(n)}$ intersects
  $F_j^{(n)}\in\sF_j\big(\conv(\sL^{(n)})\big)$. Let
  $x_{i_k}^{(n)}\in L_{i_k}^{(n)}\cap F_k^{(n)}$, $k=1,\dots,j+1$,
  and, by passing once again to subsequences, assume that, as
  $n\to\infty$, $x_{i_k}^{(n)}\to x_{i_k}^{(0)}$ for all
  $k=1,\dots,j+1$, and that $F_{j}^{(n)}\to F$. Since each
  face is a subset of an exposed face, $F_j^{(n)}$ is a subset of the
  support set $F\big(\conv(\sL^{(n)}),u^{(n)}\big)$ of
  $\conv(\sL^{(n)})$ in direction given by a unit vector $u^{(n)}$. By
  passing once again to a subsequence, assume that
  $u^{(n)}\to u^{(0)}$ as $n\to\infty$. Then,
  $F\big(\conv(\sL^{(n)}),u^{(n)}\big)$ converges to a subset of an
  exposed face $F\big(\conv(\sL^{(0)}),u^{(0)}\big)$ of
  $\conv(\sL^{(0)})$, see \eqref{eq:proof_cont_gen_pos2} and
  \eqref{eq:proof_cont_gen_pos3} above, and
  $F\subseteq F\big(\conv(\sL^{(0)}),u^{(0)}\big)$. Therefore, for all
  $k=1,\ldots,j+1$,
  \begin{displaymath}
    x_{i_k}^{(0)}\in \limsup_{n\to\infty }
    L_{i_k}^{(n)}\cap F_j^{(n)}
    =L_{i_k}^{(0)}\cap F\subseteq
    L_{i_k}^{(0)}\cap F\big(\conv(\sL^{(0)}),u^{(0)}\big),
  \end{displaymath}
  and, moreover,
  $L_{i_k}^{(0)}\cap
  F\big(\conv(\sL^{(0)}),u^{(0)}\big)=\big\{x_{i_k}^{(0)}\big\}$ by
  the imposed assumption (ii). Thus, the support set
  $F\big(\conv(\sL^{(0)}),u^{(0)}\big)$ (which is also a face because
  $\conv(\sL^{(0)})$ is a polytope) is intersected by
  $L_{i_1}^{(0)},\dots,L_{i_{j+1}}^{(0)}$. The dimension of
  $F\big(\conv(\sL^{(0)}),u^{(0)}\big)$ is not smaller than $m$ due to the
  imposed assumption (i). This completes the proof of
  \eqref{eq:cont_lemma_proof_limsup}.

  In order to finish the proof of part~\textsc{(b)} it remains to show
  that, for each $m\in\{0,1,\ldots,d-1\}$, we have
  \begin{equation*}
    \liminf_{n\to\infty} \sum_{j=m}^{d-1}\fv_m(\sL^{(n)})
    \geq \sum_{j=m}^{d-1}\fv_m(\sL^{(0)}).
  \end{equation*}
  We shall actually prove that, for each $m\in\{0,1,\ldots,d-1\}$,
  \begin{equation}
    \label{eq:liminf_cont_lemma_proof}
    \liminf_{n\to\infty} \fv_m(\sL^{(n)}) \geq \fv_m(\sL^{(0)}),
  \end{equation}
  that is, for all sufficiently large $n\in\NN$ and each
  $m$-dimensional face of $\conv(\sL^{(0)})$ there exists an
  $m$-dimensional face of $\conv(\sL^{(n)})$, which is hit by exactly
  $m+1$ sets from $\{L_1^{(n)},\ldots,L_l^{(n)}\}$. While, in view of
  \eqref{eq:proof_continuity11}, \eqref{eq:proof_continuity3b} and the
  imposed general position condition, the latter looks quite
  plausible, the rigorous proof is rather involved.
  
  Fix an $m$-dimensional face $F_m$ of $\conv(\sL^{(0)})$ for some
  $m\in\{0,1,\dots,d-1\}$ and $u\in N(\conv(\sL^{(0)}),F_m)$. Then
  $F_m$ is a subset of the $(d-1)$-dimensional hyperplane $H$
  orthogonal to $u$.  By the assumptions (i) and (ii)
  $\sM(\sL^{(0)},F_m)$ is a collection of $m+1$ sets from $\sL^{(0)}$,
  which for simplicity is assumed to be
  $\{L_{1}^{(0)},\dots,L_{m+1}^{(0)}\}$ and such that
  $\{x_{j}\}:=L_{j}^{(0)}\cap F_m$, $j=1,\dots,m+1$. Note that $F_m$
  is the convex hull of affinely independent points
  $\{x_{1},\dots,x_{m+1}\}\subset H$ and the normal cone
  $N(L_j^{(0)},x_j)\supset N(\conv(\sL^{(0)}),x_j)$ contains $u$, for
  all $j=1,\ldots,m+1$. If $m<d-1$, extend this set to
  $\{x_1,\dots,x_{d}\}$ by adding arbitrary fictitious points
  $\{x_{m+1},\ldots,x_{d}\}$ from $H$ in such a way that the points
  $\{x_1,\dots,x_{d}\}$ are affinely independent. For every fictitious
  point added, introduce a fictitious set by letting 
  $\widetilde{L}_j^{(n)}:=\widetilde{L}^{(0)}_j:=\{x_j+tu:\;
  t\in[-c,0]\}$, $j=m+2,\dots,d$, $n\in\NN$, where $c>0$ is an
  arbitrary positive constant. Finally, put
  \begin{displaymath}
    \widetilde{L}_j^{(n)}=L_{j}^{(n)}\quad\text{and}\quad
    \widetilde{L}_j^{(0)}:=L_{j}^{(0)}\quad\text{for}\quad j=1,\ldots,m+1,
  \end{displaymath}  
  and denote by $\widetilde{F}_m$ the convex hull of $\{x_1,\dots,x_{d}\}$.
  
  The reason behind introducing these fictitious objects is the
  following. Recall that our goal is to construct an $m$-dimensional
  face of $\conv(\sL^{(n)})$ which is hit by (and only by)
  $L_j^{(n)}$, for $j=1,\ldots,m+1$. We shall construct this face as
  an intersection of the convex hull $\conv(\sL^{(n)})$ with its
  appropriate supporting hyperplane, say, $\widetilde{H}_n$. If
  $m<d-1$, neither such a face nor a supporting hyperplane
  $\widetilde{H}_n$ are unique. By introducing fictitious points we
  remove these degrees of freedom and construct the supporting
  hyperplane $\widetilde{H}_n$ in a more or less straightforward
  way. Then we ``forget'' about fictitious points and sets and show
  that the constructed $\widetilde{H}_n$ possesses the required
  properties.
  
  Take an arbitrary point $z$ from the relative interior of $\widetilde{F}_m$
  and fix a sufficiency small $\eps>0$ such that:
  \begin{itemize}
  \item $\eps$ is smaller than the distance from $z$ to the
    relative boundary of $\widetilde{F}_m$;
  \item for arbitrary $y_i\in B_\eps(x_i)$,
    $i=1,\dots,d$, the points $\{y_1,\dots,y_{d}\}$ are
    affinely independent.
  \end{itemize}
  Clearly, $z$ belongs to the convex hull of $\{y_1,\dots,y_{d}\}$ for
  arbitrary $y_i\in B_\eps(x_i)\cap H$, $i=1,\dots,d$.
  Furthermore, $z$ does not belong to the convex hull of any strict
  subfamily of sets from
  $\big\{B_\eps(x_1)\cap H,\dots,B_\eps(x_{d})\cap
  H\big\}$. Indeed, the convex hull of any such subfamily lies in the
  $\eps$-neighbourhood of the relative boundary of
  $\widetilde{F}_m$ and, therefore, does not contain $z$. For the rest
  of the proof the chosen $u,z$ and the hyperplane $H$ remain
  fixed. Put $D_\eps=D_\eps(u):=H+B_{\eps}(0)$.
  Note that each ball $B_\eps(x_j)$, $j=1,\dots,d$, is a subset of
  $D_\eps$. 
  
  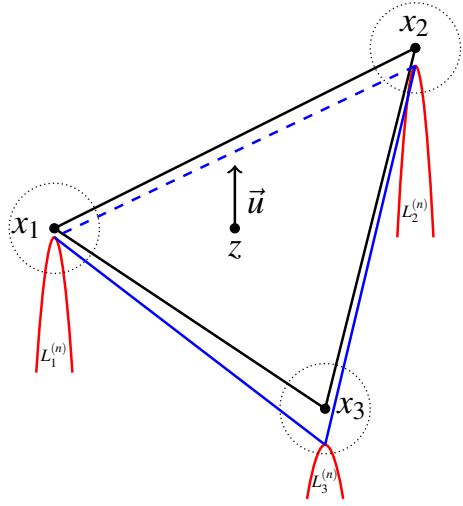
\begin{figure}[!hbtp]
    \centering
    \scalebox{1.2}{\begin{tikzpicture}
        \draw[fill=black]  (2,2) circle (0.05);
        \draw[fill=black]  (1,-2) circle (0.05);
        \draw[fill=black]  (-2,0) circle (0.05);
        
        \node[left] (c) at (-2,0) {$x_1$};
        \node[above] (c) at (2,2) {$x_2$};
        \node[right] (c) at (1,-2) {$x_3$};

        \draw[fill=black]  (0,0) circle (0.05);
        \node[below] (c) at (0,0) {$z$};

        \draw[thick,->] (0,0) -- (0,0.7);
        \node[right] (c) at (0,0.3) {$\vec{u}$};

        \draw [thick,red] plot [smooth, tension=1] coordinates {(1.8,-0.1) (2.0,1.8) (2.2,-0.1)};
        \draw [thick,red] plot [smooth, tension=1] coordinates {(0.8,-3.0) (1.0,-2.4) (1.2,-3.0)};
        \draw [thick,red] plot [smooth, tension=1] coordinates {(-2.2,-1.6) (-2,-0.1) (-1.8,-1.6)};
        \node[scale=0.5] (c) at (2.0,0.2) {$L_2^{(n)}$};
        \node[scale=0.5] (c) at (1.0,-2.8) {$L_3^{(n)}$};
        \node[scale=0.5] (c) at (-2.0,-1.4) {$L_1^{(n)}$};


        \draw[thick] (1,-2) -- (2,2);
        \draw[thick] (2,2) -- (-2,0);
        \draw[thick] (1,-2) -- (-2,0);

        \draw [thick,blue] (2.0,1.8) -- (1.0,-2.4) ;
        \draw [thick,dashed,blue] (2.0,1.8) -- (-2,-0.1) ;
        \draw [thick,blue] (-2,-0.1) -- (1.0,-2.4) ;

        \draw[densely dotted]  (2,2) circle (0.5);
        \draw[densely dotted]  (1,-2) circle (0.5);
        \draw[densely dotted]  (-2,0) circle (0.5);
        
      \end{tikzpicture}}
    \caption{Graphical illustration for the proof of relation
      \eqref{eq:liminf_cont_lemma_proof} for $d=3$ and $m=2$. A face
      $F_m=\widetilde{F}_m=\conv\{x_1,x_2,x_3\}$ of $\conv(\sL^{(0)})$
      is contained in a hyperplane $H$ passing through $x_1,x_2,x_3$
      and having a normal vector $u$; $z$ is a point in the relative
      interior of $\widetilde{F}_m$. The sets $L_1^{(n)}$, $L_2^{(n)}$
      and $L_3^{(n)}$ converge to the limiting sets (not depicted),
      which intersect $\widetilde{F}_m$ at $x_1,x_2$ and $x_3$,
      respectively. The blue triangle is the sought face of
      $\conv(\sL^{(n)})$ which is obtained as a convex hull of
      appropriate points from $L_1^{(n)}\cap B_{\eps}(x_1)$,
      $L_2^{(n)}\cap B_{\eps}(x_2)$ and $L_3^{(n)} \cap B_{\eps}(x_3)$
      for a sufficiently small $\eps>0$.}
    \label{fig_1}
  \end{figure}

  Let $A_{\eps}$ be the subset of the unit sphere formed by all
  vectors which are unit normals to the hyperplanes spanned by $d$
  affinely independent points $y_1,\ldots,y_d$ such that
  $y_j\in B_{\eps}(x_j)$, $j=1,\ldots,d$. Note that $A_{\eps}$ shrinks
  to $\{u\}$ as $\eps\downarrow 0$. From assumption (ii) we infer that
  the normal cone $N(L^{(0)}_j,x_j)$ has a non-empty interior for
  $j=1,\ldots,m+1$.  This follows from the fact that
  $N(L^{(0)}_j,x_j)\supseteq N(\conv(\sL^{(0)}),x_j)$ and $x_j$ is a
  vertex of the polytope $\conv(\sL^{(0)})$. The normal cone of the
  fictitious set $\widetilde{L}^{(0)}_j$ at $x_j$, for
  $j=m+2,\ldots,d$, has this property by construction. The reverse
  spherical image $\tau(L,\cdot)$, defined by \eqref{eq:tau_def}, is
  continuous for every compact convex set $L$, see
  \cite[Lemma~2.2.12]{schn2}. Thus, by decreasing the chosen $\eps>0$
  we can ensure that
  \begin{equation}\label{eq:tau_a_eps}
    \tau(\widetilde{L}^{(0)}_j,A_{\eps_1})
    =\{x_j\}\quad\text{for all}\quad j=1,\dots,d
    \quad\text{and}\quad\eps_1\in(0,\eps).
  \end{equation}
  
  Since for all $j=1,\dots,m+1$ the set $L_{j}^{(n)}$ converges to
  $L_{j}^{(0)}$ in the Hausdorff metric as $n\to\infty$, see
  \eqref{eq:proof_continuity11}, and the intersection of sets is upper
  semicontinuous, we conclude that $L_{j}^{(n)}$ hits
  $B_{\eps/2}(x_{j})$ and $L_{j}^{(n)}\cap D_{\eps/2}$ is a subset of
  $(L_j^{(0)}\cap D_{\eps/2})+B_{\eps/2}(0)$ for all $j=1,\dots,m+1$
  and all $n\geq n_0$ for a sufficiently large $n_0\in\NN$. Recall
  that $\widetilde{L}_j^{(n)}$ is set to be equal to
  $\widetilde{L}_j^{(0)}$ for all $j=m+2,\dots,d$, which means that
  the above claims trivially hold for $j=m+2,\dots,d$. Furthermore,
  from the inclusion
  \begin{displaymath}
    \widetilde{L}_{j}^{(n)}\cap D_{\eps/2}
    \subseteq \widetilde{L}_j^{(0)}\cap D_{\eps/2}+B_{\eps/2}(0)
    \subseteq B_{\eps}(x_j),\quad j=1,\ldots,d\quad n\geq n_0,
  \end{displaymath}
  and the choice of $\eps>0$ it follows that, for arbitrary
  $y_j\in \widetilde{L}_{j}^{(n)}\cap D_{\eps/2}$, $j=1,\ldots,d$, the
  points $\{y_1,\ldots,y_d\}$ are affinely independent.
  
  Let $\bar{L}^{(n)}$ be the convex hull of
  $\big\{\widetilde{L}_{1}^{(n)}\cap
  D_{\eps/2},\dots,\widetilde{L}_{d}^{(n)}\cap D_{\eps/2}\big\}$, and
  consider the closed segment $[z-\eps u,\,z+\eps u]$. Since the
  projections of $\widetilde{L}_j^{(n)}\cap D_{\eps/2}$ on $H$ are
  subsets of $B_{\eps}(x_j)\cap H$, the projection of $\bar{L}^{(n)}$
  onto $H$ contains $z$. Thus, the segment $[z-\eps u,\,z+\eps u]$
  intersects the boundary of $\bar{L}^{(n)}$, and no point from this
  segment is a convex combination of points from any strict subfamily
  of
  $\big\{\widetilde{L}_{1}^{(n)}\cap D_{\eps/2},\dots,
  \widetilde{L}_{d}^{(n)}\cap D_{\eps/2}\big\}$, since otherwise, $z$
  would have been such a combination. Define $y:=z+t_0u$, where
  $t_0=\sup\big\{a\in\R:\; y+au\in\bar{L}^{(n)}\big\}$.

  Pick a unit vector $v$ from $N(\bar{L}^{(n)},y)$ and note that by
  construction $v\in A_{\eps}$, and, in particular,
  $\tau(\widetilde{L}_j^{(0)},v)={x_j}$ for all $j=1,\ldots,d$ in view
  of \eqref{eq:tau_a_eps}. Clearly, $y\in F(\bar{L}^{(n)},v)$.  Let us
  show that also $y\in F(\bar{\bar{L}}^{(n)},v)$, where
  $\bar{\bar{L}}^{(n)}:=\conv\big(\widetilde{L}_1^{(n)}\cup\cdots\cup
  \widetilde{L}_{d}^{(n)}\big)$. Assume that $y$ does not belong to
  $F(\bar{\bar{L}}^{(n)},v)$.  In this case,
  $F(\bar{\bar{L}}^{(n)},v)$ would not be a subset of $D_{\eps/2}$.
  Since $F(\widetilde{L}_j^{(n)},v)=\widetilde{L}_j^{(n)}\cap H_n$,
  where
  $H_n:=\big\{x\in\R^d:\; \langle
  x,v\rangle=h(\widetilde{L}_j^{(n)},v)\big\}$, and the intersection
  operation is upper semicontinuous, we have that
  \begin{displaymath}
    F(\widetilde{L}_j^{(n)},v)
    \subset F(\widetilde{L}_j^{(0)},v)+B_{\eps/2}(0)
    =\tau(\widetilde{L}_j^{(0)},v)+B_{\eps/2}(0)
    =B_{\eps/2}(x_j)\subseteq D_{\eps/2}.
  \end{displaymath}
  Note that
  \begin{displaymath}
    F(\bar{\bar{L}}^{(n)},v)\subset\conv\bigg(\bigcup_{j=1}^d
      F(\widetilde{L}_j^{(n)},v)\bigg),
  \end{displaymath}
  and, therefore, $F(\bar{\bar{L}}^{(n)},v)\subset D_{\eps/2}$, which
  is a contradiction.  The last display also implies that
  $y\in F(\bar{\bar{L}}^{(n)},v)$ is a convex combination
  $\sum_{j=1}^dc_j y_j$ of points $y_j\in F(\widetilde{L}^{(n)}_j,v)$,
  $j=1,\dots,d$. Since $F(\widetilde{L}^{(n)}_j,v)\subset D_{\eps/2}$,
  we have also that $y_j\in \widetilde{L}^{(n)}_j\cap
  D_{\eps/2}$. Hence, the weights $c_1,\dots,c_d$ are
  strictly positive, because $y$ does not belong to the convex hull of
  any strict subfamily of
  $\big\{\widetilde{L}_{1}^{(n)}\cap
  D_{\eps/2},\dots,\widetilde{L}_{d}^{(n)}\cap D_{\eps/2}\big\}$.
  Since $y$ belongs to the support set $F(\bar{\bar{L}}^{(n)},v)$, we
  have $h(\bar{\bar{L}}^{(n)},v)=\langle y,v\rangle$ and, therefore,
  using that the support function of the convex hull is equal to the
  maximum of support functions of the involved sets,
  \begin{displaymath}
    \max_{j=1,\dots,d} h(\widetilde{L}_j^{(n)},v)
    =h(\bar{\bar{L}}^{(n)},v)=\langle y,v\rangle
    =\sum_{j=1}^{d} c_j \langle y_j,v\rangle
    \leq \sum_{j=1}^{d} c_j  h(\widetilde{L}_j^{(n)},v).
  \end{displaymath}
  This is only possible if
  \begin{displaymath}
    h(\bar{\bar{L}}^{(n)},v)=h(\widetilde{L}_{1}^{(n)},v)
    =\cdots=h(\widetilde{L}_{d}^{(n)},v).
  \end{displaymath}
  Thus, the hyperplane
  $\widetilde{H}_n:=\big\{x\in\R^d:\; \langle
  x,v\rangle=h(\tilde{L}^{(n)},v)\big\}$ intersects
  $L_{1}^{(n)},\dots,L_{m+1}^{(n)}$ at their support sets in direction
  $v$. Put
  \begin{displaymath}
    F_{m}^{(n)}:=\conv\big\{L_1^{(n)}\cap \widetilde{H}_n,
    \ldots,L_{m+1}^{(n)}\cap \widetilde{H}_n\big\},
  \end{displaymath}
  and note that the sets on the right-hand side are affinely
  independent singletons. Since all other sets in
  $\{L_{m+2}^{(n)},\ldots,L_{l}^{(n)}\}$ lie in the open half-space
  $\widetilde{H}_n^{-}$ for all sufficiently large $n\in\NN$, we
  conclude that $F_{m}^{(n)}$ is an $m$-dimensional face of
  $\conv(\sL^{(n)})$.

  Summarising, we have shown the existence of an $m$-dimensional face
  $F_m^{(n)}$ of $\conv(\sL^{(n)})$ which intersects $B_\eps(x_{j})$,
  $j=1,\dots,m+1$, and so these faces are different for different
  faces $F_m$ of $\conv(\sL^{(0)})$.  This finishes the proof of
  \eqref{eq:liminf_cont_lemma_proof}.

  \vspace{2mm}	
  \noindent
  {\sc Proof of part \textsc{(c)}}. Note that $\fv_{d-1}(\sL^{(n)})$
  is not larger than the number of $(d-1)$-dimensional faces of
  $\conv(\sL^{(n)})$.  Consider a $d$-tuple of sets
  $\{L^{(n)}_{i_1},\dots,L^{(n)}_{i_d}\}$ which contributes to
  $\fv_{d-1}(\sL^{(n)})$ and so generates at least one
  $(d-1)$-dimensional face $F_{d-1}^{(n)}$ of $\conv(\sL^{(n)})$. If
  it generates another $(d-1)$-dimensional face $\bar{F}^{(n)}_{d-1}$
  for infinitely many $n\in\NN$, then, arguing as in the proof of
  \eqref{eq:cont_lemma_proof_limsup} above, we conclude that there is
  a $d$-tuple of sets $\{L^{(0)}_{i_1},\dots,L^{(0)}_{i_d}\}$ from
  $\sL^{(0)}$ which generates two $(d-1)$-dimensional faces of
  $\conv(\sL^{(0)})$. Since $m=d-1$, the limiting face
  $F\big(\conv(\sL^{(0)}),u^{(0)}\big)$, constructed in the proof
  \eqref{eq:cont_lemma_proof_limsup} as the limit of
  $(d-1)$-dimensional faces of $\conv(\sL^{(n)})$, is of dimension
  $d-1$. This contradicts condition (ii) imposed on $\sL^{(0)}$ unless
  the limits of $F_{d-1}^{(n)}$ and $\bar{F}^{(n)}_{d-1}$ (in the
  Hausdorff metric) are identical and constitute a $(d-1)$-dimensional
  face $F_{d-1}^{(0)}$ of $\conv(\sL^{(0)})$. Each $(d-1)$-dimensional
  face is an exposed one, and so the faces $F_{d-1}^{(n)}$ and
  $\bar{F}^{(n)}_{d-1}$ arise as intersections of $\conv(\sL^{(n)})$
  with two hyperplanes $H^{(n)}$ and $\bar{H}^{(n)}$,
  respectively. Denote by $u^{(n)}$ and $\bar{u}^{(n)}$ the unit
  normals to the faces $F_{d-1}^{(n)}$ and $\bar{F}^{(n)}_{d-1}$,
  respectively. Then
  \begin{displaymath}
    F_{d-1}^{(n)}\subset \{x\in H^{(n)}:\langle x,\bar{u}^{(n)}\rangle
    \leq h(\bar{F}^{(n)}_{d-1},\bar{u}^{(n)})\}=:G^{(n)},\quad
    \bar{F}_{d-1}^{(n)}\subset \{x\in \bar{H}^{(n)}:\langle x,u^{(n)}\rangle
    \leq h(F^{(n)}_{d-1},u^{(n)})\}=:\bar{G}^{(n)}.
  \end{displaymath}
  Note that $G^{(n)}$ (respectively, $\bar{G}^{(n)}$) is a subset of
  $H^{(n)}$ (respectively, $\bar{H}^{(n)}$) with the boundary
  $H^{(n)}\cap \bar{H}^{(n)}$. Both $H^{(n)}$ and $\bar{H}^{(n)}$
  converge, as $n\to\infty$, to a $(d-1)$-dimensional hyperplane
  $H^{(0)}$, which is a supporting hyperplane of
  $\conv(\sL^{(0)})$. The intersection of $H^{(n)}$ and
  $\bar{H}^{(n)}$ is a $(d-2)$-dimensional affine subspace which
  converges to a limit $H'$, and so $G^{(n)}$ and $\bar{G}^{(n)}$
  converge to two subsets of $H^{(0)}$ bounded by $H'$. Since
  $F_{d-1}^{(n)}$ and $\bar{F}^{(n)}_{d-1}$ have an identical limit as
  $n\to\infty$ and these faces have disjoint relative interiors, the
  limit $F^{(0)}_{d-1}$ is a subset of $H'$. This is a contradiction,
  since the dimension of $H'$ is equal to $d-2$.
\end{proof}

Without assuming strict convexity of sets from $\sL^{(n)}$, the
conclusion of Lemma~\ref{lemma:conv-tuples} is wrong. For instance,
this is the case if $L_1^{(n)}$ and $L_2^{(n)}$ are collinear segments
converging to two singletons $L_1$ and $L_2$. The limiting collection
is in general position, which is not the case for
$\big\{L_1^{(n)},L_2^{(n)}\big\}$.

\begin{proof}[Proof of Theorem~\ref{thm:f_vector_convergence}]
  We shall use the Skorokhod representation theorem, see
  \cite[Theorem~4.30]{Kallenberg:2002} in conjunction with
  Theorem~\ref{th:convergence-pp} and
  Lemma~\ref{lemma:conv-tuples}. First, we can use the Skorokhod
  representation theorem to pass to a new probability space such that
  convergence in Theorem~\ref{th:convergence-pp} holds almost
  surely. On this new probability space with probability one all the
  assumptions of Lemma~\ref{lemma:conv-tuples} hold for the point
  processes $L_i^{(n)}:=n^{-1}(K-\xi_i)^{o}$, $i=1,\ldots,n$,
  $n\in\NN$, with the limit, as $n\to\infty$, given by the point
  process $\sL^{(0)}$ composed of $L_i^{(0)}:=[0,x_i]$, $x_i\in\Pi_K$,
  where for simplicity we kept the original notation for the objects
  on the new probability space. Thus, on this new probability space
  there exists a (random) $n_0\in\NN$ such that
  $\Fv(\sL_{\Xi_n})=\Fv(Z^{o})$ for all $n\geq n_0$ with probability
  one. Going back to the original probability space, we get the
  required convergence in distribution. 
\end{proof}

\subsection{Proof of Theorem~\ref{thm:f_vector_convergence_moments}}
\label{sec:proof-theor-refthm:f}

We exploit the same approach as in the proof of Theorem 2.4 in
\cite{kab:mar:tem:19}. In view of
Theorem~\ref{thm:f_vector_convergence}, it suffices to check the
uniform integrability, which is equivalent to
\begin{displaymath}
  \sup_{n\in\NN}\, \E \fv_k^m(Q_n)<\infty
\end{displaymath}
for all $k=0,\ldots,d-1$ and $m\in\NN$. By
Corollary~\ref{cor:combinatorial}, the latter is equivalent to
\begin{equation}\label{eq:moments_proof1}
  \sup_{n\in\NN}\,\E \fv_0^m(Q_n)<\infty
\end{equation}
for all $m\in\NN$, since $\binom{n}{k}\leq n^k$, $k=0,\ldots,d-1$.

By Proposition~\ref{prop:def_gp_strictly_convex}(ii),
\begin{align*}
  \fv_0(Q_n)=\fv_0\big(\conv(\sL_{\Xi_n})\big)
  &=\sum_{i=1}^{n}\one_{\{(K-\xi_i)^{o}\text{ is a vertex of the family } \sL_{\Xi_n}\}}\\
  &\leq \sum_{i=1}^{n}\one_{\{(K-\xi_i)^{o}\text{ does not lie in the convex hull of }(K-\xi_j)^{o},\;j=1,\ldots,n,\; i\neq j\}}.
\end{align*}
Let $(\eta_n)_{n\in\NN}$ be a sequence of independent copies of $\xi$
which is also independent of $(\xi_n)_{n\in\NN}$. Raising both sides
of the last display to the power $m$ and taking expectations we see
that \eqref{eq:moments_proof1} follows, once we check that
\begin{displaymath}
  p_n:=n^{m}\P\bigg\{\text{for all }j=1,\ldots,m,\;(K-\eta_j)^o
    \not\subset\conv\Big(\bigcup_{i=1}^n(K-\xi_i)^o\Big)\bigg\}
  =n^m\P\bigg(\bigcap_{j=1}^m\{K-\eta_j
    \not\supseteq X_n\}\bigg)=\mathcal{O}(1),
\end{displaymath}
as $n\to\infty$, for every fixed $m\in\NN$, where the constant in the
Landau symbol may depend on $m$.

Put
\begin{displaymath}
  \chi_n:=\inf\{t\geq 0:\; tK\supseteq X_n\},
\end{displaymath}
and note that $\chi_n\in(0,1]$. Using this variable we can bound
$p_n$ as follows:
\begin{align}
  p_n&\leq n^m \E\left[\Prob{K-\eta_j\not\supseteq
       \chi_nK\text{ for all }j=1,\ldots,m \Big| X_n}\right]\notag\\
     &=n^m \E\bigg(1-\frac{V_d(K\ominus\chi_n K)}{V_d(K)}\bigg)^m
       =n^m \E\Big[\big(1-(1-\chi_n)^d\big)^m\Big]\notag\\
     &\leq d^m n^m \E \chi_n^m=d^m n^m  \int_0^1\Prob{t^{1/m}K\not\supset X_n}{\rm d}t \notag\\
     &=d^m n^m \int_0^1\Prob{K^o\not\subseteq t^{1/m}X^o_n}{\rm d}t=d^m n^m \int_0^1\Prob{K^o\not\subset tX^o_n} mt^{m-1} {\rm d}t.\label{eq:moments_proof2}
\end{align}

We shall now derive an appropriate upper bound for
$\Prob{K^o\not\subset t X_n^o}$, which is uniform in $t\in(0,\,1]$. To
this end, we recall some concepts from convex geometry. The
exoskeleton of $K$ is the set $\exo(K)$ of points $x\in\Int K$ such
that $x$ does not have a unique nearest point from $\partial K$. Note
that $\exo(K)$ has vanishing $d$-dimensional Lebesgue measure, see,
for example, \cite[p.~106]{kid:rat06}. For all
$x\in\Int K\setminus\exo(K)$, define the projection map
$p(K,x)\label{eq:proj_map_def}$, which associates with $x$ the closest
point from $\partial K$. Write $u(K,x)$ for the unit vector
$(p(K,x)-x)/\rho(\partial K,x)$, where
$\rho(\partial K,x)=\|p(K,x)-x\|$ denotes the distance from $x$ to the
set $\partial K$. Unlike \cite[Chapter~4]{schn2}, where these concepts
are used for $x$ outside $K$, we employ them for $x$ from the interior
of $K$.

Consider a supporting hyperplane $H(K,p(K,x))$. It is apparent that
this is also the supporting hyperplane to the ball
$B_{\rho(\partial K,x)}(x)$ touching the boundary of $K$ at
$p(K,x)$. Thus, $u(K,x)$ belongs to the normal cone $N(K,p(K,x))$ and
$p(K,x)$ belongs to the support set $F(K,u(K,x))$.

For a set $R\subset\Sphere$ and $t\geq 0$, put
\begin{displaymath}
  T_K(R,t):=\big\{x\in\Int K\setminus\exo(K):\;
  p(K,x)\in \tau(K,R),\; \rho(\partial K,x)\leq t\big\},
\end{displaymath}
where $\tau(K,R)$ is the reverse spherical image of a set $R$ defined
at \eqref{eq:tau_def}.

By \cite[Theorem~1]{kid:rat06} applied with $C=\tau(K,R)$, $A=K$,
$P=B=W=\{0\}$, $Q=B_1(0)$ and $\eps=t$, we have that
\begin{equation}
  \label{eq:moments_proof3}
  \lim_{t\to0}\; t^{-1} V_d\big(T_K(R,t)\big)=S_{d-1}(K,R).
\end{equation}
Further, for $R\subseteq \Sphere$ and $s\geq 0$, denote
\begin{displaymath}
  \hat{R}(s):=\big\{x\in\R^d:\; x/\|x\|\in R,\, \|x\|\geq s\big\}.
\end{displaymath} 

From Lemma~\ref{lem:disjoint_r} presented after this proof we see
that there exist $M\in\NN$, $\eps>0$ and a finite disjoint family
$R_1,\dots,R_M\subseteq \Sphere$ such that
\begin{enumerate}[(i)]
\item for all $j=1,\ldots,M$ we have $S_{d-1}(K,R_j)>0$;
\item $B_{\eps}(0)\subset \conv\{y'_1,\ldots,y'_M\}$ for arbitrary
  $y'_j\in R_j$, $j=1,\ldots,M$.
\end{enumerate}
If $y_j\in \hat{R}_j(1)$, then $y'_j:=y_j/\|y_j\|\in R_j$,
for $j=1,\ldots,M$, so that 
\begin{displaymath}
  \conv\{y_1,\ldots,y_M\}\supseteq\conv\{y'_1,\ldots,y'_M,0\}
  =\conv\{y'_1,\ldots,y'_M\}\supseteq B_{\eps}(0).
\end{displaymath}

For $j=1,\ldots,n$, put 
\begin{displaymath}
  \zeta_j:=u(K,\xi_j)/\rho(\partial K,\xi_j),
\end{displaymath}
and note that
\begin{displaymath}
  (K-\xi_1)^o\supseteq [0,\zeta_1].
\end{displaymath}
Indeed, $(K-\xi_1)^o\supseteq [0,\,\zeta_1]$ if and only if
$K-\xi\subset [0,\,\zeta_1]^{o}$, and $[0,\,\zeta_1]^{o}$ is a
half-space $H^{-}_{u(K,\xi_1)}\big(\rho(\partial K,\xi_1)\big)$, which, by
definition of $u(K,\xi_1)$ and $\rho(\partial K,\xi_1)$, contains
$K-\xi_1$. Further, note that with probability one $\xi_1\in T_K(R,t)$
if and only if $\zeta_1/\|\zeta_1\|\in R$ and
$\|\zeta_1\|\geq t^{-1}$, that is, $\zeta_1\in \hat{R}(t^{-1})$. Here
we have used that $p(K,\xi_1)\in\tau(K,R)$ if and only if
$u(K,\xi_1)\in R$.

We are now in position to derive a uniform upper bound on 
$\Prob{K^o\not\subset t X_n^o}$. Pick $a>0$ so large that
$K^o\subset B_a(0)$. Choose $R_1,\ldots,R_M$ and $\eps>0$ satisfying
(i) and (ii) above. By construction, if
$\{\zeta_1,\dots,\zeta_n\}\cap
\hat{R}_j(t^{-1}\eps^{-1}a)\neq\varnothing$ for all $j=1,\dots,M$,
then
\begin{displaymath}
  B_{a}(0)\subset t\conv\{\zeta_1,\dots,\zeta_n\}.
\end{displaymath}
Since
$$
\conv\{\zeta_1,\dots,\zeta_n\}\subseteq X_n^{o},
$$
we obtain
\begin{multline*}
  \P\big\{K^o\not\subset tX_n^o\big\}
  \leq \P\big\{B_{a}(0)\not\subset t\conv\{\zeta_1,\dots,\zeta_n\}\big\}
  \leq \sum_{j=1}^M\Prob{\{\zeta_1,\dots,\zeta_n\}\cap \hat{R}_j(t^{-1}\eps^{-1}a)=\emptyset}\\
  = \sum_{j=1}^M \Big(1-\P\big\{\zeta_1\in\hat{R}_j(t^{-1}\eps^{-1}a)\big\}\Big)^n
  = \sum_{j=1}^M \Big(1-\P\big\{\xi_1\in T_K(R_j,t\eps a^{-1})\big\}\Big)^n.
\end{multline*}
Using \eqref{eq:moments_proof3} and monotonicity of $V_d(T_K(R,t))$,
whenever $S_{d-1}(K,R)>0$, there exists a constant $c_0=c_0(R)>0$ such
that
\begin{displaymath}
  \frac{V_d\big(T_K(R,t)\big)}{t}\geq c_0,\quad t\in (0,\,\eps a^{-1}].
\end{displaymath}
Therefore, 
\begin{displaymath}
  \P\big\{\xi_1\in T_K(R_j,t\eps a^{-1})\big\}
  =\frac{V_d\big(T_K(R_j,t\eps a^{-1})\big)}{V_d(K)}
  \geq \frac{\min_{j=1,\ldots,M}c_0(R_j)}{V_d(K)}t\eps a^{-1}
  =:c'_0t,\quad t\in (0,\,1],\quad j=1,\ldots,M,
\end{displaymath}
where $c'_0>0$, and, thereupon,
\begin{displaymath}
  \P\big\{K^o\not\subset tX_n^o\big\}\leq M(1-c'_0t)^n,\quad t\in (0,\,1].
\end{displaymath}
From \eqref{eq:moments_proof2} we finally obtain
\begin{align*}
  p_n\leq
  d^m n^m \int_0^1\P\big\{K^o\not\subset tX^o_n\big\} mt^{m-1} {\rm d}t
  &\leq d^m n^m mM \int_0^1 (1-c'_0 t)^n t^{m-1} {\rm d}t\\
  &=d^m mM\int_0^n \Big(1-\frac{c'_0s}{n}\Big)^n s^{m-1}{\rm d}s
  \leq d^m mM\int_0^\infty e^{-c'_0s}s^{m-1}{\rm d}s<\infty
\end{align*}
for all $n\in\NN$. The proof is complete. \qed

\begin{lemma}\label{lem:disjoint_r}
  Let $S_{d-1}(K,\cdot)$ be the surface area measure of a convex body
  $K$. Then there exists a finite family of disjoint Borel sets
  $R_1,\dots,R_M$ on the unit sphere and $\eps>0$, such that
  $S_{d-1}(K,R_j)>0$ for all $j=1,\dots,M$ and, for all points
  $y_j\in R_j$, $j=1,\ldots,M$, the convex hull of $\{y_1,\dots,y_M\}$
  contains the ball $B_\eps(0)$.
\end{lemma}
\begin{proof}
  Denote by $S_K$ the support of $S_{d-1}(K,\cdot)$, so that $S_K$ is
  a closed subset of the unit sphere $\Sphere$.  It is well known,
  see, for example, \cite[Section~8.2.1]{schn2}, that the measure
  $S_{d-1}(K,\cdot)$ has its centroid at the origin, that is,
  $\int_{\Sphere} u S_{d-1}(K,{\rm d}u)=0$. Furthermore, $S_K$ is not
  a subset on any great subsphere of $\Sphere$. Hence, $\conv(S_K)$
  contains a ball $B_{3\eps}(0)$ for a sufficiently small $\eps>0$.

  Let $(P_n)_{n\in\NN}$ be a sequence of polytopes with vertices in
  $S_K$ such that $P_n$ converges to $\conv(S_K)$ in the Hausdorff
  metric as $n\to\infty$. Take $n_0\in\NN$ so large that
  $B_{2\eps}(0)\subseteq P_{n_0}$. Let $z_1,\ldots,z_M$ be the
  vertices of $P_{n_0}$, so that
  $P_{n_0}=\conv\{z_1,\ldots,z_M\}$. Pick $\delta>0$ such that the
  balls $B_{\delta}(z_j)$ are disjoint for $j=1,\ldots,M$ and put
  $R_j(\delta):=B_{\delta}(z_j)\cap S_K$.  Since
  $\conv\{z_1,\ldots,z_M\}\ominus B_{\delta}(0)$ converges to
  $P_{n_0}$ in the Hausdorff metric as $\delta\downarrow 0$, it is
  clear that we can further choose $\delta_0>0$ so small that
  $B_{\eps}(0)\subseteq \conv\{z_1,\ldots,z_M\}\ominus
  B_{\delta_0}(0)$. Thus, for an arbitrary choice of
  $y_j\in R_j(\delta_0)=:R_j$, $j=1,\ldots,M$, we have
  $B_{\eps}(0)\subseteq \conv\{y_1,\ldots,y_M\}$. Since $R_j$ is a
  relative neighbourhood of a point $z_j\in S_K$, we have
  $S_{d-1}(K,R_j)>0$ for all $j=1,\ldots,M$.
\end{proof}

\subsection{Limit theorems for the number of \texorpdfstring{$K$}{K}-facets}

In this subsection additionally to strict convexity and regularity we
also assume that $K$ is a generating set. The latter is needed to
ensure applicability of Lemmas~\ref{lemma:facets} and \ref{lemma:fd}.
  
Recall that, in general, $\fv_{d-1}(Q_n)$ can be strictly smaller than
the number of $K$-facets of $Q_n$, see
Example~\ref{ex:two-points-2}. Still, for the limiting polytope $Z^o$
in Theorem~\ref{thm:f_vector_convergence}, the number of facets
$f_{d-1}(Z^o)$ coincides with the $(d-1)$-st component of the
$\Fv$-vector for the family of segments $\{[0,x]:x\in\Pi_K\}$. By
Lemma~\ref{lemma:facets}, the number of $K$-facets of $Q_n$ coincides
with the number of $(d-1)$-dimensional faces of
$\conv(\sL_{\Xi_n})$. Further, by Lemma~\ref{lemma:conv-tuples}(\textsc{c}) the
latter is equal to $\fv_{d-1}(\sL_{\Xi_n})$ for all $n\geq n_0$, where
$n_0\in\NN$ is random. Therefore, the number of $K$-facets of $Q_n$
converges in distribution to $f_{d-1}(Z^{o})$ as $n\to\infty$.
  
In order to ensure the uniform integrability of the number of
$K$-facets, we impose the following property on $K$.  A strictly
convex body $K\in\sK^d$ is said to satisfy a \emph{finite boundary
  intersection property} if there is a finite number $C_K$ such that
the cardinality of the intersection of
$\partial K+x_1,\dots,\partial K+x_d$ is at most $C_K$ for Lebesgue
almost all $x_1,\dots,x_d\in\R^d$. This property can be equivalently
formulated as the fact that for Lebesgue almost all sets
$\{x_1,\dots,x_d\}\subset\R^d$, there are at most $C_K$ different
translations of $K$ which have these points on the boundary. It is
easy to see that Euclidean balls and ellipsoids have this property
with $C_K=2$. The same is the case for all strictly convex bodies in
the plane, see \cite{good:wood81}.  However, it is possible to
construct examples of bodies which do not have a finite intersection
property. We conjecture, however, that all origin symmetric strictly
convex bodies have a finite boundary intersection property.

The finite boundary intersection property makes it possible to bound
the number of $K$-facets of $Q_n$ in terms of the relevant component
of the $\Fv$-vector. Summarising, we obtain the following corollary.
  
\begin{corollary}
  \label{cor:K-facets-limit}
  Assume that $K\in\sK^d_{(0)}$ is strictly convex, regular and is
  also a generating set. Then the number of $K$-facets of $Q_n$
  converges in distribution to $f_{d-1}(Z^o)$ as $n\to\infty$. If $K$
  satisfies a finite boundary intersection property, then all power
  moments of the number of $K$-facets of $Q_n$ converge to the
  corresponding moments of $f_{d-1}(Z^o)$ as $n\to\infty$. In
  particular, the expected number of $K$-facets of $Q_n$ converges, as
  $n\to\infty$, to the constant given at the right-hand side of
  \eqref{eq:f_0_Z_gen2}.  If $K$ is also origin symmetric, this
  constant simplifies to $2^{-d}d! V_d(\piv K)V_d((\piv K)^o)$, where
  $\piv K$ is the projection body of $K$.
\end{corollary}
\begin{proof}
  The stated convergence in distribution has been already explained
  above. For the convergence of moments we argue as follows. If $K$
  satisfies a finite boundary intersection property, then, following
  the proof of Lemma~\ref{lemma:fd}, we see that each $d$-tuple of
  sets from $\sL_{\Xi_n}$ intersects at most $C_K$ of
  $(d-1)$-dimensional faces of $\conv(\sL_{\Xi_n})$. Hence, the number
  of $K$-facets of $Q_n$ is at most $C_K\fv_{d-1}(Q_n)$. The
  convergence of all moments follows now from the uniform
  integrability of $(\fv_{d-1}^m(Q_n))_{n\in\NN}$ for all $m\in\NN$.
\end{proof}

\subsection{Application to ball convex sets}
\label{sec:appl-ball-con}

Assume that $K$ is the unit Euclidean ball $B_1(0)$. In this case, the
limit of $nX_n$ is the zero cell $Z$ of a stationary isotropic Poisson
hyperplane tessellation. The Poisson process $\Pi_{B_1(0)}$ has
intensity measure with density proportional to $\|x\|^{-(d+1)}$,
$x\in\R^d\setminus\{0\}$, and its convex hull $Z^o$ is the polar set
to $Z$. 

In the isotropic case, the constants $\E f_k(Z)$ have been calculated
for $k=0$ and $k=d-1$ in \cite{kab:mar:tem:19}, see Theorem~2.4 and
Remark~2.5 therein; and for arbitrary $k$ in \cite{kab:20}, see
Theorem~2.1 therein. The next result follows from
Theorems~\ref{thm:f_vector_convergence}
and~\ref{thm:f_vector_convergence_moments} together with
Corollary~\ref{cor:K-facets-limit}.

\begin{corollary}
  Assume that $K$ is a unit ball in $\R^d$.  Then
  \begin{equation}
    \label{eq:mom_convergence_isotropic}
    \Fv(Q_n)=\Fv(\sL_{\Xi_n})\dodn {\bf f}\big(\conv(\Pi_{B_1(0)})\big)
    \quad \text{as }\; n\to\infty, 
  \end{equation}
  and also the number of $K$-facets of $Q_n$ converges in distribution
  to $f_{d-1}\big(\conv(\Pi_{B_1(0)})\big)$.  We also have the convergence of
  power moments of all orders. In particular, the average number of
  $K$-facets of $Q_n$ converges, as $n\to\infty$, to
  \begin{equation}
    \label{eq:fodors_result}
    \E f_{d-1}\big(\conv(\Pi_{B_1(0)})\big)=2^{-d}d! \kappa_d^2,
  \end{equation}
  where $\kappa_d=\pi^{d/2}/\Gamma(1+d/2)$ is the volume of the
  $d$-dimensional unit ball.
\end{corollary}

The convergence of the expected number of $K$-facets to a constant
given by \eqref{eq:fodors_result} has been proved for $d=2$ in
\cite{fod:kev:vig14}\footnote{Actually, the result has been proved for
  the average number of vertices but it is easy to see that for $d=2$
  the number of vertices and edges ($K$-facets) are the same, see
  p.~903 in \cite{fod:kev:vig14}.}. If $d=2$ the limiting constant is
$\pi^2/2$.

In our work, the limiting constant $2^{-d} d!\kappa_d^2$ in
\eqref{eq:fodors_result} appears in a somehow implicit way as a
consequence of \eqref{eq:f_vec_convergence} and the uniform
integrability.  It would be nice to have \eqref{eq:fodors_result}
confirmed using direct calculations as has been done in
\cite{fod:kev:vig14} in two dimensions. An attempt towards this goal
has been made in the preprint \cite{fod19}, which, however, seems to
remain incomplete up to date. However, the priority in discovering the
correct constant in \eqref{eq:fodors_result} should be given to
\cite{fod19}, where this constant appears in its first version.

\section{Appendix}
\label{sec:appendix}

\subsection{Some properties of random samples from a convex
  body}

The aim of this part is to show that the family $\sL_{\Xi_n}$ is in
general position with probability one.

\begin{lemma}\label{lem:gen_pos}
  Assume that a convex body $K\in\sK^{d}$ is strictly convex and
  regular.  Let $\xi_1,\xi_2,\ldots,\xi_{d+1}$ be independent copies
  of a random variable $\xi$ with the uniform distribution on
  $K$. Then
  \begin{displaymath}
    \Prob{\text{there exists }x\in\R^d\text{ such that
      }\; \{\xi_1,\xi_2,\ldots,\xi_{d+1}\}\subset (\partial K - x)}=0.
  \end{displaymath}
  Furthermore, if $1\leq m\leq d$ and $\eta$ is a random vector in
  $\R^d$ such that $\{\xi_1,\ldots,\xi_m\}\subset \partial K-\eta$
  a.s., then
  \begin{displaymath}
    \P\Big\{\text{the one-dimensional normal cones } N(K-\eta,\xi_i),
      \;i=1,\ldots,m, \text{ are linearly independent}\Big\}=1.
  \end{displaymath}
\end{lemma}
\begin{proof}
  We start with the second statement and use the results of
  \cite{rat:zah05} about transversal intersection of Lipschitz
  manifolds. Note that $\partial K$ is a Lipschitz manifold since the
  boundary of $K$ is $C^1$. Furthermore, since $K$ is convex we can
  work with usual normal cones instead of Clarke cones used in
  \cite{rat:zah05}, see \cite[Proposition~2.4.4]{Clarke:1990}.  The
  normal cones $N(K,\eta+\xi_1)$ and $N(K,\eta+\xi_2)$ are
  one-dimensional and different with probability one. Furthermore,
  $N(K,\eta+\xi_2)=-N(K,\eta+\xi_1)$ with probability zero. Indeed,
  this equality holds only if $\eta+\xi_2$ is equal to the support point of
  $K$ in direction $-N(K,\eta+\xi_1)$, which is a singleton.
  Therefore, with probability one the Lipschitz manifolds
  $\partial K+\xi_i$, $i=1,\dots,d$, intersect transversally, see
  \cite[Section~6]{rat:zah05}. By Lemma~6 of this cited work, there
  exists an $m$-dimensional linear subspace of the linear hull of
  $N(K,\eta+\xi_i)$, $i=1,\dots,m$. In particular, this means that
  these normal cones are linearly independent.
  
  Consider the random set
  \begin{displaymath}
    Y:=\bigcap_{i=1}^d (\partial K-\xi_i).
  \end{displaymath}
  From the above proof with $m=d$, we see that, for almost all
  realisations of $\xi_1,\dots,\xi_d$ and each $y\in Y$, the convex
  hull of the normal cones $N(K-y,\xi_1),\dots,N(K-y,\xi_d)$ has
  nonempty interior in $\R^d$.  It is obvious that
  $Y\subset\partial\bar Y$, where
  \begin{displaymath}
    \bar Y:=\bigcap_{i=1}^d (K-\xi_i)=K\ominus\{\xi_1,\dots,\xi_d\}.
  \end{displaymath}
  At any $y\in Y$, the normal cone $N(\bar Y,y)$ is the convex hull of
  the normal cones $N(K-\xi_i,y)$, $i=1,\dots,d$, and so is of full
  dimension in $\R^d$. Thus, strict convexity and regularity of $K$
  yield that $S_{d-1}\big(K, N(\bar Y,y)\big)>0$. Since the cones $N(\bar Y,y)$
  are different for different $y$, we deduce that the set $Y$ is at
  most countable.

  Then $\{\xi_1,\xi_2,\ldots,\xi_{d+1}\}\subset (\partial K - x)$ if
  and only if $x\in Y$ and $\xi_{d+1}+x\in \partial K$. The
  probability that such an $x$ exists is at most
  $\Prob{\xi_{d+1}+Y\cap \partial K\neq\varnothing}$. This probability
  vanishes, since the distribution of $\xi_{d+1}$ is absolutely
  continuous, $\xi_{d+1}$ is independent of $Y$ and $Y$ is at most
  countable. Alternatively, the first statement can be derived by
  checking that
  \begin{displaymath}
    \P\bigg\{\dim_{\mathrm{H}}\Big(\bigcap_{i=1}^{d}(\partial
    K-\xi_i)\Big)=0\bigg\}=1,
  \end{displaymath}
  using Theorem~13.12 and Corollary~8.11 from \cite{Mattila:1999}, where
  $\dim_{\mathrm{H}}$ denotes the Hausdorff dimension. Note that this proof does
  not require regularity nor strict convexity of $K$, which results in
  a weaker statement that $Y$ has the Hausdorff dimension zero instead
  of being at most countable.
\end{proof}

\subsection{Vague convergence of measures on the family of convex
  compact sets}\label{sec:vague-conv-meas}

Let $X$ be a random convex set in $\sK^d_0\setminus\{0\}$, that is,
$X$ a.s.\ contains the origin. Its $n$ independent copies constitute a
binomial point process denoted by $\Psi_n$.

\begin{theorem}
  \label{thr:pp}
  Let $(\Psi_n)_{n\in\NN}$ be a sequence of binomial processes on
  $\sK^d_0\setminus\{0\}$, and let $\Psi$ be a locally finite Poisson
  process on $\sK^d_0\setminus\{0\}$. Then $n^{-1}\Psi_n$ 
  converges in distribution to $\Psi$ if and only if $n^{-1}Z_n$
  converges in distribution to a
  random compact convex set $Z$ as $n\to\infty$, where $Z_n$
  (respectively, $Z$) is the convex hull of the union of the sets from
  $\Psi_n$ (respectively, $\Psi$).
\end{theorem}
\begin{proof}
  Denote the intensity measure of the limit process $\Psi$ by $\mu$,
  and let $\Psi_n:=\{X_1,\dots,X_n\}$ consist of $n$ independent
  copies of a random convex set $X$ with distribution $\nu$. Note that
  both $\mu$ and $\nu$ are measures on $\sK^d_0\setminus\{0\}$.

  It is well known (as a simple version of the Grigelionis theorem for
  general binomial processes, see, e.g.,
  \cite[Proposition~11.1.IX]{dal:ver08} or
  \cite[Corollary~4.25]{kalle17} or \cite[Theorem~4.2.5]{mo1}) that
  $n^{-1}\Psi_n$ converges in distribution  to $\Psi$ if and only if
  $\mu_n(\cdot):=n\nu(n\cdot)$ vaguely converges to $\mu$ on
  $\sK^d_0\setminus\{0\}$ as $n\to\infty$. In other words,
  \begin{equation}
    \label{eq:10v}
    n\Prob{n^{-1}X\in\sA }\to \mu(\sA)\quad \text{as }\; n\to\infty 
  \end{equation}
  for all $\sA\in\sB_0$ and such that $\sA$ is a continuity set for
  $\mu$.

  Introduce subfamilies of $\sK^d_0\setminus\{0\}$ by letting
  \begin{displaymath}
    \sA_L:=\{A\in\sK^d_0\setminus\{0\}: A\subset L\},
  \end{displaymath}
  where $L\in\sK^d_0\setminus\{0\}$ is an arbitrary compact convex
  set containing the origin and which is distinct from $\{0\}$. We
  first prove that the vague convergence $\mu_n\to\mu$ follows from
  \eqref{eq:10v} with $\mu$-continuous sets of the form $\sA_L^c$
  taken instead of general $\sA$.

  Fix an $\eps>0$ and let $L_0:=B_\eps(0)$ be the closed centred ball
  of radius $\eps$. It is always possible to ensure that $\sA_{L_0}^c$
  is a continuity set for $\mu$. For each $\sA\in\sB_0$, let
  \begin{equation}
    \label{eq:tildemu}
    \tilde{\mu}_n(\sA):=\frac{\mu_n(\sA\cap\sA_{L_0}^c)}
    {\mu_n(\sA_{L_0}^c)},\quad  n\geq 1,
  \end{equation}
  and define $\tilde{\mu}$ by the same transformation applied to
  $\mu$.  Then $\tilde{\mu}_n$ is a probability measure on
  $\sK^d_0\setminus\{0\}$ and so on $\sK^d$.

  It is known that $\tilde{\mu}_n$ converges in distribution to
  $\tilde{\mu}$ if and only if
  $\tilde{\mu}_n(\sA_L)\to\tilde{\mu}(\sA_L)$ for all $L\in\sK^d$ such
  that $\sA_L$ is a continuity set for $\tilde{\mu}$ and
  $\tilde{\mu}(\sA_L)\uparrow 1$ if $L$ increases to the whole space,
  see \cite[Theorem~1.8.14]{mo1}. The latter is clearly the case,
  since $\Psi$ has a locally finite intensity measure, hence, at most
  a finite number of its points intersects the complement of $B_r(0)$
  for any $r>0$.
    
  It obviously suffices to assume in \eqref{eq:10v} that $\sA$ is
  closed in the Hausdorff metric. Then there exists an $\eps>0$ such
  that each $A\in\sA$ is not a subset of $B_\eps(0)=:L_0$. Then
  $\sA\cap\sA_{L_0}^c=\sA$, so that
  $\tilde{\mu}_n(\sA)=\mu_n(\sA)/\mu_n(\sA_{L_0}^c)$ and
  $\tilde{\mu}(\sA)=\mu(\sA)/\mu(\sA_{L_0}^c)$. Finally, note that the
  convergence of the denominator in \eqref{eq:tildemu} follows from
  \eqref{eq:10v} for $\sA=\sA_{L_0}^c$ and recall that $L_0$ is chosen
  so that $\sA_{L_0}$ is $\mu$-continuity set.

  Therefore, it is possible to check \eqref{eq:10v} only for
  $\sA=\sA_L^c$, meaning that $n^{-1}\Psi_n$ converges in distribution
  to $\Psi$ if and only if
  \begin{equation}
    \label{eq:Ps1}
    n\Prob{n^{-1}X\not\subset L}\to \mu(\sA_L^c)\quad \text{as }\;
    n\to\infty  
  \end{equation}
  for all $L\in\sK^d_0\setminus\{0\}$ such that $\sA_L$ is a
  continuity set for $\mu$.  

  By \cite[Theorem~1.8.14]{mo1}, $n^{-1}Z_n$ converges in distribution to
  $Z$ if and only if
  \begin{equation}
    \label{eq:Ps2}
    \Prob{n^{-1}Z_n\subset L}\to \Prob{Z\subset L}\quad \text{as }\;
    n\to\infty  
  \end{equation}
  for all $L\in\sK^d_0$ such that $L$ is a continuity set for $Z$,
  that is, $\Prob{Z\subset L}=\Prob{Z\subset\Int L}$, and
  $\Prob{Z\subset L}\uparrow 1$ as $L$ increases to the whole
  space. The latter condition is the case by the assumed compactness of
  $Z$. Since
  \begin{displaymath}
    \Prob{Z\subset L}=\exp\{-\mu(\sA_L^c)\},
  \end{displaymath}
  $L$ is a continuity set for $Z$ if and only if $\sA_L$ is a
  continuity set for $\mu$. 
  
  Finally, note that
  \begin{displaymath}
    \Prob{n^{-1}Z_n\subset L}=\Big(1-\Prob{n^{-1}X\in\sA_L^c}\Big)^n,
  \end{displaymath}
  so that \eqref{eq:Ps1} is equivalent to \eqref{eq:Ps2}.
\end{proof}

\subsection{The expected number of vertices in the zero cell of the
  anisotropic Possion tessellation}

Recall that the zero cell $Z$ is the intersection of all half-spaces
$H^{-}_{u_i}(t_i)$, where $\mathcal{P}_{K}=\{(t_i,u_i):i \geq 1\}$ is
the Poisson process on $(0,\infty)\times \Sphere$ introduced in
Subsection \ref{subsec:X_n_conv}. The next theorem provides a formula
for the expected number $\E f_0(Z)$ of vertices of the random polytope
$Z$.

Let $H_i:=H_{u_i}(t_i)$ be the boundary of $H^{-}_{u_i}(t_i)$. Denote
by $\widehat{\mu}$ the intensity measure of the Poisson hyperplane
process $\{H_i:i\geq 1\}$ on the affine Grassmannian $A(d,d-1)$ of all
$(d-1)$-dimensional affine subspaces of $\R^d$.

\begin{theorem}\label{thm:zero_cell_vertexes}
  Let $Z$ be the zero cell of the anisotropic Poisson tessellation
  induced by the hyperplane process $\{H_{u_i}(t_i):i\geq 1\}$. Then
  formula \eqref{eq:f_0_Z_gen2} holds true.
\end{theorem}
\begin{proof}
  Without loss of generality we may and do assume that $V_d(K)=1$. We
  start by noticing that
  \begin{equation}\label{app:eq:h_def} 
    h(\piv K,x):=\frac{1}{2}\int_{\Sphere}|\langle x,u\rangle|
    S_{d-1}(K,{\rm d}u)=
    \int_{\Sphere}\langle x,u\rangle
    \one_{\{\langle x,u\rangle\geq 0\}}S_{d-1}(K,{\rm d}u),
  \end{equation}
  which follows from the relation $a\one_{\{a\geq 0\}}=(a+|a|)/2$,
  $a\in\R$, and the fact that $\int_{\Sphere}u S_{d-1}(K,{\rm
    d}u)=0$. Hence, $h(\piv K,x)$ is indeed the support function of the
  projection body $\piv K$ of $K$, see \cite[Eq.~(5.80)]{schn2}. Note
  also that $h(\piv K,x)$ is equal to the $\widehat{\mu}$-content of
  the set of $H\in A(d,d-1)$ such that $x\notin H^-$.
  
  Let $f$ be an arbitrary nonnegative measurable function. By
  repeating verbatim the proof given in the Appendix of \cite{sch82},
  it can be checked that
  \begin{displaymath}
    \int_{(A(d,d-1))^d} f(x)
    \one_{\{H_1\cap\cdots\cap H_d=\{x\}\}}
    {\rm d}\widehat{\mu}(H_1)\cdots {\rm d}\widehat{\mu}(H_d)\\
    =\int_{\R^d} f(x) J(x){\rm d}x,
  \end{displaymath}
  where $J(x)$ is given at \eqref{eq:f_0_Z_gen3}. Using this equality
  with $f(x):=e^{-h(\piv K,x)}$ and the multivariate Mecke equation,
  see \cite[Cor.~3.2.3]{sch:weil08}, we obtain
  \begin{align*}
    \E f_0(Z)
    &=\frac{1}{d!}\sum_{i_1\geq 1,\ldots,i_d\geq 1}
      \one_{\{H_{i_1},H_{i_2},\ldots,H_{i_d}\text{ intersect at a vertex of }Z\}}\\
    &=\frac{1}{d!}\int_{(A(d,d-1))^d}\P\{H_1,\ldots,H_d
      \text{ intersect at a vertex of }
      Z\cap H_1\cap\cdots\cap H_d\}
      {\rm d}\widehat{\mu}(H_1)\cdots {\rm d}\widehat{\mu}(H_d)\\
    &=\frac{1}{d!}\int_{(A(d,d-1))^d}e^{-h(\piv K,x)}
      \one_{\{H_1\cap\cdots \cap H_d=\{x\}\}}
      {\rm d}\widehat{\mu}(H_1)\cdots {\rm d}\widehat{\mu}(H_d)\\
    &=\frac{1}{d!}\int_{\R^d} e^{-h(\piv K,x)} J(x){\rm d}x.
  \end{align*}
  Passing to the polar coordinates and using that
  $h(\piv K,tu)=th(\piv K,u)$ and $J(tu)=J(u)$ for all $t>0$ and
  $u\in\Sphere$, we obtain
  \begin{multline*}
    \E f_0(Z)=\frac{1}{d!}\int_0^{\infty}\int_{\Sphere}
    e^{-th(\piv K,u)}J(u)t^{d-1}{\rm d}t{\rm d}u\\
    =\frac{1}{d!}\int_{\Sphere}J(u)\int_0^{\infty}
    e^{-s}s^{d-1}(h(\piv K,u))^{-d}{\rm d}s{\rm d}u
    =\frac{1}{d}\int_{\Sphere}(h(\piv K,u))^{-d}J(u){\rm d}u.
  \end{multline*}
  The proof is complete.
\end{proof}

\begin{center}{\sc Acknowledgements}\end{center}

The work of both authors was supported by a grant IZHRZ0\_180549 from
the Swiss National Science Foundation and Croatian Science Foundation,
project ``Probabilistic and analytical aspects of generalised regular
variation''. The work of AM has also received funding from the Ulam
Program of the Polish National Agency for Academic Exchange (NAWA),
project No. PPN/ULM/2019/1/00004/DEC/1. Both authors are grateful to
the University of Wroclaw for hospitality.

The authors are grateful to Peter Kevei who triggered this work by
drawing their attention to the research on ball hulls, and to Ferenc
Fodor and Daniel Hug for further discussions in the course of this
work. The authors are indebted to Vlad Bohun for his assistance in
making the simulations. We also thank the referee for a number of
useful remarks and comments.

\let\oldaddcontentsline\addcontentsline
\renewcommand{\addcontentsline}[3]{}
\bibliographystyle{abbrv}
\bibliography{Spindle}
\let\addcontentsline\oldaddcontentsline
\end{document}